%% file: main.tex
\documentclass[11pt]{article} 
\usepackage[margin=1in]{geometry}

\input{preamble}

\newcommand*\samethanks[1][\value{footnote}]{\footnotemark[#1]}

\begin{document}

\title{Optimal score estimation via empirical Bayes smoothing}
\author{Andre Wibisono\thanks{Department of Computer Science, Yale
University}, ~ 
Yihong Wu\thanks{Department of Statistics and Data Science, Yale
University \\
\indent\indent \texttt{andre.wibisono@yale.edu}, ~ \texttt{yihong.wu@yale.edu}, ~ \texttt{yingxi.yang@yale.edu}}, ~ and
Kaylee Yingxi Yang\samethanks[2]}

\date{}
\maketitle

\begin{abstract}
  We study the problem of estimating the score function of an unknown probability distribution $\rho^*$ from $n$ independent and identically distributed observations in $d$ dimensions. Assuming that $\rho^*$ is subgaussian and has a Lipschitz-continuous score function $s^*$, we establish the optimal rate of $\tilde \Theta(n^{-\frac{2}{d+4}})$ for this estimation problem under the loss function $\|\hat s - s^*\|^2_{L^2(\rho^*)}$ that is commonly used in the score matching literature, highlighting the curse of dimensionality where sample complexity for accurate score estimation grows exponentially with the dimension $d$. Leveraging key insights in empirical Bayes theory as well as a new convergence rate of smoothed empirical distribution in Hellinger distance, we show that a regularized score estimator based on a Gaussian kernel attains this rate, shown optimal by a matching minimax lower bound. We also discuss extensions to estimating $\beta$-H\"older continuous scores with $\beta \leq 1$, as well as the implication of our theory on the sample complexity of score-based generative models.
\end{abstract} 

\tableofcontents

\section{Introduction}

Sampling from a probability distribution is a fundamental algorithmic task in many applications; for example, in Bayesian statistics, we draw samples from the posterior distribution to perform approximate inference.
The {\em score function} of a distribution, which is defined as the derivative of the logarithm of the density of the distribution, encodes rich information about the distribution.
In particular, if we have access to the score function of a distribution, then we can sample from it by running any first-order sampling algorithm such as the Langevin dynamics or the Hamiltonian Monte Carlo.
Recent results have shown mixing time guarantees for such algorithms under structural assumptions on the target distribution, such as log-concavity or isoperimetry; see for example~\cite{dalalyan2017further,durmus2019analysis,bou2020coupling}.
More recently, an alternative method for sampling known as the ``Score-based Generative Models'' (SGMs) have been proposed, which operates via following the reverse diffusion process from a standard distribution such as the standard Gaussian to the target distribution; see for example~\cite{song2019generative, song2020score, HJA20}.
Implementing SGMs as an algorithm requires approximating the score function of the target distribution along the forward diffusion process; this can be done for example via score matching~\cite{hyvarinen2005estimation}, and in practice this is typically trained via neural networks.
A recent wave of theoretical results has shown that assuming we have access to a good sequence of score estimators along the forward process with provable error guarantees, then algorithms derived from SGMs have good mixing time guarantees, with the same or even better iteration complexity as classical algorithms such as based on the Langevin dynamics, but without requiring assumptions such as isoperimetry, log-concavity, or even smoothness on the target distribution; see for example~\cite{LLT22a,lee2023convergence,CCL+22,benton2023linear}.

Motivated by the wide applications of the score function, in this paper we study the problem of estimating the score function of a probability distribution from independent samples.
We assume the target distribution has full support on $\R^d$, is subgaussian, and has a Lipschitz-continuous score function. The Lipschitzness of the score function is a common assumption in sampling literature, including for analyzing classical Langevin-based algorithms~\cite{dalalyan2017further, dalalyan2017theoretical,  cheng2018convergence, durmus2019highdimensional, dalalyan2019user, durmus2019analysis, vempala2019rapid} and the newly developed SGMs~\cite{BMR20, DBTHD21, LLT22a, WY22, CLL2022, lee2023convergence, CCL+22}.
Let  $\calP_{\alpha, L}$ denote the class of probability distributions on $\R^d$ that are $\alpha$-subgaussian with $L$-Lipschitz score functions; here we assume $\alpha^2 L \ge 1$ to ensure that $\calP_{\alpha, L}$ is not empty.
Let $\rho^*$ be a probability density in $\calP_{\alpha, L}$.
The \textbf{score function} of $\rho^*$ is the vector field $s^\ast \colon \R^d \to \R^d$ defined by
\begin{align*}
    s^\ast(x) = \nabla \log \rho^*(x)
\end{align*}
for all $x \in \R^d$, where $\nabla$ is the gradient with respect to $x$.
Observing samples $\X_n=(X_1,\dots,X_n)$ drawn i.i.d.\ (independently and identically distributed) from $\rho^*$, our goal is to learn a \textbf{score estimator} $\hat s(\cdot) \coloneqq \hat s(\cdot;\X_n)$ that uses the samples to approximate the true score function $s^\ast$.
We measure the score estimation error using the following loss:
\begin{align}
\label{eq:LossFunction}
    \loss(\hat s, \rho^*) \coloneqq \|\hat s - s^\ast\|_{\rho^*}^2 \,=\, \int_{\R^d} \|\hat s(x) - s^\ast(x)\|^2 \rho^*(x)dx.
\end{align}
There are a number of reasons why this is a meaningful loss function for score estimation.
\begin{itemize}
    \item The loss function~\eqref{eq:LossFunction} is the relevant error metric in the application of score matching as assumed in the recent works in SGMs. For example, \cite{CCL+22} showed that 
    the sampling error of 
    a popular type of SGM known as 
    Denoising Diffusion Probabilistic Modeling 
    \cite{HJA20}
    can be bounded up to the score matching loss \eqref{eq:LossFunction} and discretization error.
    We will revisit this in \prettyref{sec:application-sgm} and
    discuss the implication of our results on SGMs.
    
\item 
If the estimator $\hat s$ is {\em proper}, i.e.\ it is the score function of a valid density $\hat \rho$, then \eqref{eq:LossFunction} equals the \textit{relative Fisher information} (or \textit{Fisher distance}) \cite[Eq.~(9.25)]{villani2021topics} between $\rho^*$ and $\hat \rho$. 
In this work, however, we do not limit the scope to proper score estimators.

\item In the special case when $\rho^\ast$ is a Gaussian mixture, the loss function \prettyref{eq:LossFunction} is precisely the \textit{regret}, the central quantity in the theory of \textit{empirical Bayes} (EB), that measures the excess risk of a data-driven procedure over the Bayesian oracle risk. Although our model class is far richer than Gaussian mixtures, this connection with empirical Bayes is crucial for developing our score estimator; see \prettyref{sec:idea}.

    \item Finally, it is also necessary to consider a squared loss weighted by the true density $\rho^*$ 
    as the score cannot be estimated well in low-density regions.
        Indeed, for the unweighted squared loss 
    $\|\hat s - s^\ast\|_2^2 = \int_{\R^d} \|\hat s(x) - s^\ast(x)\|^2 dx$,        
        it is easy to show that the minimax score estimation error is infinite.\footnote{For example, the densities $\rho_0=\calN(0,1)$ and 
    $\rho_1=\calN(\mu_n,1)$ are statistically indistinguishable with sample size $n$ for $\mu_n=2^{-n}$, but their score functions $s_0$ and $s_1$ differ by a constant, and hence $\|s_0-s_1\|_2=\infty$.}
\end{itemize}

The minimax risk of score estimation over the density class $\calP_{\alpha, L}$ under the loss function \prettyref{eq:LossFunction} with sample size $n$ is defined as
\begin{align}
\label{eq:MinimaxRisk}
    \calR_n(\calP_{\alpha, L}) \coloneqq \inf_{\hat{s}} \sup_{\rho^* \in \calP_{\alpha, L}} \E\loss(\hat s, \rho^*) 
\end{align}
where the expectation is over $\X_n = (X_1,\dots,X_n) \sim (\rho^*)^{\otimes n}$
and the infimum is taken over all estimators $\hat{s}$ that is measurable with respect to $\X_n$. 
Taking a step toward understanding the theoretical aspects of score estimation, we summarize the major contributions of this paper as follows:
\begin{enumerate}
    \item  We study a regularized score estimator based on the KDE (kernel density estimator) using Gaussian kernel.
    We analyze the performance of this estimator $\hat s$ (see~\eqref{Eq:KernelEst0}) under the loss function~\eqref{eq:LossFunction} 
    and establish an upper bound on the minimax risk as follows:    
    \begin{align}
    \label{eq:minimax-upper-bound}
        \sup_{\rho^* \in \mathcal{P}_{\alpha,L}}\E\loss (\hat s, \rho^*)  \lesssim n^{-\frac{2}{d+4}} \, \polylog(n).
    \end{align}
    When $\rho^*$ is such that the score function $s^*$ is $(L,\beta)$-H\"older continuous for some $0 < \beta \le 1$, using the same estimator, we extend the upper bound to the following:
    $$\E\loss (\hat s, \rho^*) \lesssim n^{-\frac{2\beta}{d+2\beta+2}} \polylog(n).$$
    
    \item We prove a matching minimax lower bound:
    \begin{align}
    \label{eq:minimax-lower-bound}
    \calR_n(\calP_{\alpha, L})
    \gtrsim n^{-\frac{2}{d+4}}
    \end{align}
    thereby showing that the optimal rate of score estimation is 
    $n^{-\frac{2}{d+4}}$ up to logarithmic factors.     
    The proof adapts the standard approach for establishing minimax lower bounds in nonparametric density estimation using Fano's lemma with modifications made for scores.
    Comparing the lower bound~\eqref{eq:minimax-lower-bound} with the upper bound~\eqref{eq:minimax-upper-bound}, we observe the typical ``curse of dimensionality'' which suggests that to achieve a specified level of accuracy in score estimation, the sample complexity must increase exponentially with dimension.
    
    \item We discuss some implications of our results in the context of SGMs. In particular, we propose a regularized score estimator along the Ornstein-Uhlenbeck (OU) process targeting standard Gaussian, which is the usual forward process in SGMs.
    For estimating the score function at time $t$ along the forward process, by using intermediate results in the proof of the upper bound~\eqref{eq:minimax-upper-bound}, we derive an error bound of $\Tilde{O}\inparen{\step^{-d/2}(tn)^{-1}}$ in the weighted squared loss, where $\step$ is the step size in the SGM algorithm.
\end{enumerate}

\subsection{Main idea}
\label{sec:idea}
Let us discuss the algorithm that attains the optimal rate \prettyref{eq:minimax-upper-bound} and the broad strokes of its analysis in connection to the empirical Bayes theory.
Let $\hat \rho = \frac{1}{n} \sum_{i=1}^n \delta_{X_i}$ denote the empirical distribution of the sample and its Gaussian smoothed version:
\begin{align}
    \hat \rho_\bdw = \hat \rho * \N(0, \bdw I_d)=
    \frac{1}{n} \sum_{i=1}^n \N(X_i, \bdw I_d).
    \label{eq:smooth-emp}
\end{align}
for some bandwidth parameter $\bdw > 0$. 
This Gaussian smoothing is an algorithmic device that allows us to tap into the powerful machinery of empirical Bayes. 
Instead of applying the score of $\hat \rho_\bdw$, we consider its  regularized version:
\begin{align}
\label{Eq:KernelEst0}
    \hat s^{\error}_{\bdw}(x) \,\coloneqq\, \frac{\nabla \hat{\rho}_\bdw (x) }{ \max(\hat{\rho}_\bdw (x), \error)}
\end{align}
for some  regularization parameter $\error > 0$. 
A deep result in empirical Bayes due to \cite{JZ09} (see also \cite{SG2020} for extensions to multiple dimensions) is that 
the error between regularized scores of two Gaussian mixtures is upper bounded within logarithmic factors by the squared Hellinger distance between the two mixtures. Applying this result to our setting
and bounding the likelihood ratio of $\frac{\rho^*}{\rho_h^*}$ using the score smoothness, we obtain
\begin{equation}
 \|\hat s_h^\error - s_\bdw^{\ast\error}\|_{\rho_h^*}^2 \lesssim
\|\hat s_h^\error - s_\bdw^{\ast\error}\|_{\rho^*}^2
\lesssim \frac{1}{h} H^2(\hat\rho_h,\rho_h^*) \cdot \polylog\pth{\frac{1}{H^2(\hat\rho_h,\rho_h^*) }, \frac{1}{\error}}.
    \label{eq:program1}
\end{equation}
Here $\rho_h^* = \rho^* * \calN(0,hI_d)$ is the smoothed version of the true density, $s_\bdw^{\ast\error} = \frac{\nabla \rho^\ast_\bdw}{\max(\rho^\ast_\bdw, \error)} $ is its regularized score, and 
the squared Hellinger distance between two densities $p$ and $q$ is
\begin{align*}
    \H^2(p,q) \coloneqq \int_{\R^d} \inparen{\sqrt{p(x)} - \sqrt{q(x)}}^2 dx.
\end{align*}

Next, we bound the Hellinger distance between the smoothed empirical distribution and the population:
\begin{equation}
\Expect[H^2(\hat\rho_h,\rho_h^*)] \lesssim \frac{1}{n h^{d/2}}\cdot \polylog(n).
    \label{eq:program2}
\end{equation}
Crucially relying on the smoothness of the score of $\rho^*$, this result seems \textit{not} obtainable from the literature on smooth empirical distribution based only on the subgaussianity of $\rho^*$  \cite{GGJ+2020,block2022rate}. (In fact, we suspect whether \prettyref{eq:program2} is true without the score smoothness. See \prettyref{lem:hellinger-bound} and the surrounding discussion in Appendix \ref{sec:bound-hellinger} for details.)

Finally, we control the error due to smoothing and regularization:
\begin{equation}
\|s_h^{\ast\error}-s^*\|_{\rho^*}^2 
\lesssim \frac{\error}{h} \polylog\pth{\frac{1}{\error},\frac{1}{h}} + h.
    \label{eq:program3}
\end{equation}
Since the regularization parameter only contributes logarithmically in \prettyref{eq:program1}, we may choose it rather aggressively as $\error=n^{-2}$. Balancing the main terms of 
$\frac{1}{nh^{1+d/2}}$ and $h$ leading to the optimal choice of bandwidth parameter $h=n^{-\frac{2}{d+4}}$ and the optimal rate \prettyref{eq:minimax-upper-bound}.

It is helpful to clarify the difference between classical empirical  Bayes and the present paper. In EB denoising, one observes i.i.d.\ samples $Y_1,\ldots,Y_n\sim \rho_h^* = \rho^* * \N(0,h I_d)$, where the variance parameter $h$ is fixed by the problem and the prior $\rho^*$ is an \textit{arbitrary} distribution with only tail assumptions (e.g.~subgaussian). Therein, the goal is to compete with the oracle who knows the prior $\rho^*$ and computes the Bayes estimator of $X_i\sim\rho^*$ given the noisy observation $Y_i$. Thanks to the 
Tweedie's formula \cite{efron2011tweedie}
\begin{equation}
\Expect[X_i \mid Y_i = y] = y + h s_h^\ast(y),
\label{eq:tweedie0}
\end{equation}
this is equivalent to estimating the score of $s_h^*$. 
Given an approximate score $\tilde s$, the regret of the approximate Bayes denoiser $\tilde  X(y) =   y + h \tilde s (y)$, i.e., the excess risk over the Bayesian oracle applied to a fresh observation, is given by the score matching loss \prettyref{eq:LossFunction}, namely $h^2 \ell(\tilde s,\rho_h^*)$. 

A popular method in EB (the so-called $g$-modeling approach \cite{efron2014two}) is to first compute an estimate $\tilde \rho$ of the distribution $\rho^*$ based on $Y_i$'s using deconvolution techniques, such as nonparametric maximum likelihood (NPMLE) \cite{JZ09,SG2020}, then bound the density estimation error $\tilde\rho_h = \tilde \rho * \N(0,hI_d)$ in Hellinger distance, and finally the regret of the regularized score of $\tilde \rho_h$ using tools such as \prettyref{eq:program1}. 
In our setting, since we have access to samples
$X_i$'s drawn from $\rho^*$ before convolution, we can directly apply the smoothed empirical distribution with an optimized bandwidth parameter $h$.

\subsection{Related work}
\label{sec:related}
\paragraph{Empirical Bayes.} 

As mentioned earlier, for Gaussian mixtures the score matching loss \prettyref{eq:LossFunction} and the empirical Bayes regret is equivalent. This connection can be made more precise. 
Consider the nonparametric class of Gaussian mixtures $\pi * \N(0,I_d)$, where the mixing distribution $\pi$ is supported on a ball of radius $r=O(1)$. In view of \prettyref{eq:tweedie0}, this is a subset of our model class $\calP_{\alpha,L}$ for some $\alpha,L$ depending on $r$.
On this subclass, the best score estimation error is at most 
$O(\frac{(\log n)^5}{n})$ \cite{JZ09} in one dimension (see \cite{SG2020} for extensions to  $d$ dimensions), achieved by the NPMLE.
A different approach is carried out in \cite{li2005convergence} based on KDE that applies (different) polynomial kernels of logarithmic degree to estimate the density and the derivative leads to similar results with worse logarithmic factors. 
In terms of negative results, for $d=1$, a lower bound
$\Omega\inparen{\frac{(\log n)^2}{(\log \log n)^2 n}}$ is shown in \cite{PW21}.
Compared with these near-parametric rates $\tilde\Theta(n^{-1})$, the nonparametric rate $\tilde\Theta(n^{-\frac{2}{d+4}})$ in \prettyref{eq:minimax-upper-bound} is much slower, as  the class $\calP_{\alpha,L}$ we consider is much richer than Gaussian mixtures.

\paragraph{Density Estimation.} 
We can view score estimation as a density estimation problem under a different measurement of the estimation error. 
In score estimation, if $\hat{s}$ is a proper estimator, i.e., $\hat s = \nabla \log \hat \rho$ for some distribution $\hat \rho$, then the loss function~\eqref{eq:LossFunction} is the relative Fisher information: 
\begin{equation}
\FI(\rho^* \,\|\, \hat{\rho}) := \|\nabla \log \rho^* - \nabla \log \hat{\rho}\|^2_{\rho^*},
    \label{eq:FI}
\end{equation} 
which depends on both the density itself and its first-order gradient information. 
In classical density estimation, common choices of the loss function include the squared $L^2(\R^d)$ loss 
$\|\rho^*- \hat{\rho}\|_{2}^2 \coloneqq \int_{\R^d} (\rho^*(x) - \hat{\rho}(x))^2 dx$
and the squared Hellinger distance 
$\H^2(\rho^*, \hat{\rho})$.
There is a rich literature on density estimation of nonparametric Gaussian mixtures. In one dimension, the optimal $L^2(\R)$ error is known to be $\Theta((\log n)^{1/2}/n)$;
for the squared Hellinger, the lower bound is $\Omega(\log n / n)$~\cite{Ibragimov2001, Kim2014}, and the upper bound is $O((\log n)^2/n)$ achieved by the NPMLE~\cite{JZ09}.
In $d$ dimensions, the optimal $L^2(\R^d)$ error is $\Theta((\log n)^{d/2}/n)$ \cite{KG2022};
for the squared Hellinger, the lower bound is $\Omega((\log n)^{d} / n)$~\cite{KG2022}, and the upper bound is~$O((\log n)^{d+1}/n)$ which is established for the NPMLE~\cite{SG2020}.

\paragraph{Score Estimation.}
While many methods have been proposed for estimating the score function, theoretical results are scarce. One approach involves density estimation techniques, such as kernel density estimation or neural network-based methods, followed by differentiation of their logarithm; see for example~\cite{scott1992, papamakarios2017masked}. 
Another approach estimates the unnormalized log-density and differentiates this estimate; this is effective since the score function does not depend on the normalizing constant. 
A popular method in this area is called score matching~\cite{hyvarinen2005estimation}; this method proceeds by minimizing the relative Fisher information between the data distribution and the learned model distribution, which is equivalent to the loss function~\eqref{eq:LossFunction} in the case that the score estimator is proper. It has been shown that the minimization of the score matching loss~\eqref{eq:LossFunction} is a consistent estimator assuming the global minimum is found by the optimization algorithm used in the estimation~\cite{hyvarinen2005estimation}.
The work of~\cite{sutherland2018efficient} uses the Nystr\"om approximation to speed up the score matching procedure to learn an exponential family density model with the natural parameter in a reproducing kernel Hilbert space, which may be infinite-dimensional, as introduced in~\cite{sriperumbudur2017density}.
The work of~\cite{zhou2020nonparametric} studies nonparametric score estimation via kernel ridge regression and proved the sample complexity of the resulting score estimator under some assumptions, including that the underlying score function can be written as the image of an integral operator in a reproducing kernel Hilbert space.
The work of~\cite{saremi2018deep} trained a neural network to minimize the score matching objective and output the energy--unnormalized log-density. 
For a review of modern approaches to energy-based model training, see for example~\cite{song2021train}.
Besides parameter estimation in unnormalized models, one can also train a neural network to directly output the score by minimizing the score matching objective~\cite{song2020sliced}.

\paragraph{Score-based Generative Models.}
Recent advancements in SGMs have focused on convergence analyses of the algorithms assuming access to an accurate score estimator. 
Initial studies either hinged on structural assumptions on the data distribution such as a log-Sobolev inequality~\cite{LLT22a, WY22} or strong assumptions on score estimation error such as $L^\infty$ error~\cite{DBTHD21}, or they led to bounds that exponentially increased with the problem parameters~\cite{debortoli2022, BMR20}. 
Subsequent studies have achieved polynomial convergence rates under less restrictive assumptions, including that the data distribution has a finite second moment and the scores along the forward process are Lipschitz~\cite{CLL2022, lee2023convergence, CCL+22}. 
More recent results including~\cite{benton2023linear} have established polynomial convergence guarantees under a minimal assumption of a finite second moment of the data distribution.
Parallel to these developments, significant efforts have been directed towards the problem of score estimation in SGMs, including the following. 
The work of~\cite{CHZ+2023} studied the score estimation using neural networks and derived a finite-sample bound for a specifically chosen network architecture and parameters, with the assumption that the data lies in a low-dimensional linear subspace.
The work of~\cite{pmlr-v202-oko23a} bounded the estimation error when using a neural network and showed that diffusion models are nearly minimax-optimal estimators in the total variation and in the Wasserstein distance of order one, assuming the target density belongs to the Besov space.
The work of~\cite{scarvelis2023closed} proposed to smooth the closed-form score from empirical distribution to obtain an SGM that can generate samples without training.
The work of~\cite{cui2023analysis} obtained an error rate of $\Theta(1/n)$ for SGM when the target distribution is a mixture of two Gaussians and using a two-layer neural network for learning the score function. 
The work of~\cite{cole2024score} showed that the score function can be approximated efficiently via neural networks when the target distribution is subgaussian and has a log-relative density with respect to the Gaussian measure which is a Barron function, i.e.\ can be approximated efficiently by neural networks. 
The work of~\cite{li2024good} studied SGM with score estimator from empirical kernel density estimator, similar to our work; they showed the sample complexity when the target distribution is either a standard Gaussian or has bounded support, and discussed the issue of memorization of training samples. 
A concurrent work by~\cite{zhang2024minimax} shows that when the target distribution belongs to the $\beta$-Sobolev space with $\beta \le 2$, the diffusion model with a kernel-based score estimator is minimax optimal up to logarithmic factors.

\subsection{Notations and definitions}

We review the necessary notations and definitions. Let $\calP(\R^d)$ denote the space of probability distributions on $\R^d$.
For distributions $\rho, \nu \in \calP(\R^d)$ which are absolutely continuous with respect to the Lebesgue measure on $\R^d$, for convenience we also write their probability density functions as $\rho \colon \R^d \to \R$ and $\nu \colon \R^d \to \R$.
Recall the total variation (TV) distance between $\rho$ and $\nu$ is
\begin{align*}
    \TV(\rho, \nu) = \sup_{A \subseteq \R^d} |\rho(A) - \nu(A)| \,=\, \frac{1}{2} \int_{\R^d} |\rho(x) - \nu(x)| \, dx.
\end{align*}
For a function $f\colon\R^d \to \R$ and a probability distribution $\rho$ on $\R^d$, the squared $L^2(\rho)$-norm of $f$ is
\begin{align*}
    \|f\|_{\rho}^2 := \E_\rho [f^2] = \int_{\R^d} f(x)^2 \,\rho(x) \, dx.
\end{align*}
We define $L^2(\rho)$ to be the space of functions $f \colon \R^d \to \R$ for which $\|f\|_\rho^2 < \infty.$
Similarly, given a vector field $s \colon \R^d \to \R^d$, the squared $L^2(\rho)$-norm of $s$ is
\begin{align*}
    \|s\|_{\rho}^2 := \E_\rho \left[\|s\|^2\right] = \int_{\R^d} \|s(x)\|^2 \, \rho(x) \, dx.
\end{align*}
We say a probability distribution $\rho$ on $\R^d$ is \textbf{$\alpha$-subgaussian} for some $0 < \alpha < \infty$ if for all $\theta \in \R^d$:
\begin{align*}
    \E_{\rho} \exp(\theta^\top (X - \E_\rho X)) \le \exp\inparen{\frac{\alpha^2 \|\theta\|^2}{2}}.
\end{align*}
We say a random variable $X \sim \rho$ is subgaussian if its distribution $\rho$ is subgaussian.

We use $a= O(b)$ or $b=\Omega(a)$ to indicate that $a \le Cb$ for a universal constant $C > 0$. We use $a = \Theta(b)$ to indicate that $C_1 b \le a \le C_2 b$ for $C_2 > C_1 > 0$. And $\Tilde{O}(\cdot)$ hides logarithmic factors.

\section{Main results}

\subsection{Score estimator via Empirical Bayes smoothing}
\label{Sec:UpperBound}

Suppose we are given a sample of $n$ i.i.d.\ observations  $\X_n = (X_1,\dots,X_n)$ from an unknown distribution $\rho^\ast \in \calP_{\alpha, L}$. Our goal is to estimate the score 
$s^\ast =\nabla \log \rho^\ast$.

For $\bdw > 0$, let $\hat \rho_\bdw$ be the smoothed empirical distribution which is a mixture of Gaussians:
\begin{align*}
    \hat \rho_\bdw = \frac{1}{n} \sum_{i=1}^n \N(X_i, \bdw I_d).
\end{align*}
We propose the following regularized KDE score estimator $\hat s^\error_\bdw(\cdot) = \hat s^\error_\bdw(\cdot;\X_n)$:
\begin{align}
\label{Eq:KernelEst}
    \hat s^{\error}_{\bdw}(x) \,\coloneqq\, \frac{\nabla \hat{\rho}_\bdw (x) }{ \max(\hat{\rho}_\bdw (x), \error)}
\end{align}
for some bandwidth parameter $\bdw > 0$ and regularization parameter $\error > 0$. 

We measure the accuracy of the score estimator $\hat s^\error_\bdw$ in the expected square $L^2(\rho^\ast)$-norm as in~\eqref{eq:LossFunction}, and establish the following error bound on the order of $\tilde O(n^{-\frac{2}{d+4}})$.
Here the expectation is taken over the sample $\X_n \sim (\rho^\ast)^{\otimes n}$.

\begin{theorem}
\label{thm:l2-error-kde}
Let $d \ge 1$ be fixed, and suppose we have $X_1,\dots,X_n$ drawn i.i.d.\ from some $\rho^\ast \in \cP_{\alpha,L}$.
Setting
\begin{align*}
\error = n^{-2}
\qquad\text{ and }\qquad
\bdw = \inparen{\frac{d^3 (\alpha^2 \log n)^{d/2}}{L^2 n}}^{\frac{2}{d+4}},
\end{align*}
for sufficiently large $n$, the score estimator~\eqref{Eq:KernelEst} satisfies
\begin{align*}
    \E\loss(\hat s_\bdw^\error, \rho^\ast) \le C d \alpha^2 L^2 \pth{\log n}^{\frac{d}{d+4}} n^{-\frac{2}{d+4}}
\end{align*}
where $\loss(\cdot, \cdot)$ is defined in~\eqref{eq:LossFunction}, and $C >0$ is a universal constant.
\end{theorem}

\begin{proof}
We provide the main argument for proving Theorem~\ref{thm:l2-error-kde}, deferring some of the lemmas to the appendix.
We define:
\begin{equation}
\rho^\ast_\bdw = \rho^\ast * \N(0, \bdw I_d), \qquad s^\ast_\bdw = \nabla\log \rho^\ast_\bdw, \qquad \textnormal{and}\quad s_\bdw^{\ast\error} = \frac{\nabla \rho^\ast_\bdw}{\max(\rho^\ast_\bdw, \error)}.
\label{eq:def3}
\end{equation}

Since $s^\ast$ is $L$-Lipschitz, we can show that the density ratio of $\rho^\ast$ to $\rho^\ast_\bdw$ is bounded from above everywhere by a constant:
In Lemma~\ref{lmm:ratio} in Appendix~\ref{sec:preliminaries}, we show that for all $x \in \R^d$,
\begin{align*}
    \frac{\rho^\ast(x)}{\rho^\ast_\bdw(x)} \le \exp\inparen{dL\bdw/2}.
\end{align*}
Then by a change of measure from $\rho^\ast$ to $\rho^\ast_\bdw$, we get
\begin{align}
\label{eq:change-of-measure-ub}
\E\loss(\hat s^{\error}_{\bdw}, \rho^\ast) =  \E \| \hat s^{\error}_{\bdw} - s^\ast \|^2_{\rho^\ast} \le\exp\inparen{dLh/2} \E \| \hat s^{\error}_{\bdw} - s^\ast \|^2_{\rho^\ast_\bdw}.
\end{align}
We can decompose the last factor on the right-hand side above as follows:
\begin{align}
\label{eq:decomposition}
     \E \| \hat s^{\error}_{\bdw} - s^\ast \|^2_{\rho^\ast_\bdw}
    \le\; 3\E \| \hat s^{\error}_{\bdw} - s^{\ast\error}_{\bdw} \|^2_{\rho^\ast_\bdw} \,+\, 3 \| s^{\ast\error}_{\bdw} - s^{\ast}_{\bdw} \|^2_{\rho^\ast_\bdw}  +\, 3 \| s_{\bdw}^{\ast} - s^{\ast}  \|^2_{\rho^\ast_\bdw}.
\end{align}
We now bound each of the three terms above separately.

\paragraph{First term:}
The first term $\E\| \hat s^{\error}_{\bdw} - s^{\ast\error}_{\bdw} \|^2_{\rho^\ast_\bdw}$ concerns the distance between the regularized score functions of $\hat \rho_\bdw$ and of $\rho^\ast_\bdw$, which we can bound as follows:
If $\error \in (0, \, (2\pi \bdw)^{-d/2}e^{-1/2}]$ and $ \alpha^2 n^{-2/d}\log n \lesssim \bdw \leq 1/(4L)$, then by Lemma~\ref{lem:bound-c1} in Appendix~\ref{sec:pf-bound-c1},
\begin{align}
\label{eq:upper-bound-1}
\E\| \hat s^{\error}_{\bdw} - s^{\ast\error}_{\bdw} \|^2_{\rho^\ast_\bdw} \le \frac{Cd\inparen{C_{\bdw, d, \alpha}+d}}{n\bdw}\insquare{ \left( \log \frac{1}{\error(2\pi \bdw)^{d/2}}\right)^{3} + \log \frac{n}{C_{\bdw, d, \alpha}+d}}.
\end{align}
where $C >0$ is a universal constant and $C_{\bdw, d, \alpha}\coloneqq \pth{\frac{\alpha^2 \log n}{h}}^{d/2}$. The proof of Lemma~\ref{lem:bound-c1} proceeds by first bounding $\| \hat s^{\error}_{\bdw} - s^{\ast\error}_{\bdw} \|^2_{\rho^\ast_\bdw}$ in terms of the squared Hellinger distance between $\hat{\rho}_\bdw$ and $\rho^\ast_\bdw$ using the result of~\cite{SG2020}, extending \cite[Theorem 3]{JZ09} to $d$ dimensions. 
This crucially uses the regularization in the score estimates. Another crucial component of the proof involves establishing a bound for the expected Hellinger distance between Gaussian-smoothed empirical distribution to the population, which gives us the desired bound of $\frac{1}{n\bdw^{d/2}}$ and thereby achieving the optimal rate in Theorem~\ref{thm:l2-error-kde}. We formally present this result as Lemma~\ref{lem:hellinger-bound} in Appendix~\ref{sec:bound-hellinger}.

\paragraph{Second term:}
The second term $\| s^{\ast\error}_{\bdw} - s^\ast_{\bdw} \|^2_{\rho^\ast_\bdw}$ is the error induced by the regularization, which we can bound as follows: 
If $0 \leq \error \le (2\pi \bdw)^{-d/2}/e$ and 
$h \leq \alpha^2$, then
\begin{align}
\label{eq:upper-bound-3}
    \| s^{\ast\error}_{\bdw} - s^\ast_{\bdw} \|^2_{\rho^\ast_\bdw} \le 
    \frac{2\error}{h} (64 \alpha^2 \log n)^{d/2} 
    \log \frac{1}{\error (2\pi h)^{d/2}}      
    + \frac{2d^{3/2}}{hn^2}.
\end{align}
We state the result formally as Lemma~\ref{lem:bound-c2new} in Appendix~\ref{sec:pf-bound-c2}.

\paragraph{Third term:}
The third term $\| s^\ast_\bdw - s^\ast \|^2_{\rho^\ast_\bdw}$ is a bias term that we can bound using the Lipschitzness of the score function.
Concretely, by Lemma~\ref{lem:bound-c4} in Appendix~\ref{sec:pf-bound-c4}, we have: If $\bdw < 1/(4L)$,
\begin{align}
\label{eq:upper-bound-4}
\| s^\ast_\bdw - s^\ast \|^2_{\rho^\ast_\bdw} \le L^2\bdw d.
\end{align}

\paragraph{Combining the bounds.}
Combining the bounds~\eqref{eq:upper-bound-1}--\eqref{eq:upper-bound-4}, we have
\begin{align*}
    \E \| \hat s^{\error}_{\bdw} - s^\ast \|^2_{\rho^\ast_\bdw} &\le\frac{Cd\inparen{C_{\bdw, d, \alpha}+d}}{\bdw n}\left( \log \frac{1}{\error (2\pi \bdw)^{d/2}}\right)^{3} +  \frac{Cd\pth{C_{\bdw, d, \alpha}+d}}{\bdw n}\, \log \frac{n}{C_{\bdw, d, \alpha}+d}\\
    & \quad + \frac{6\error}{\bdw} (64 \alpha^2 \log n)^{d/2} 
\log \frac{1}{\error (2\pi \bdw)^{d/2}}+ \frac{6d^{3/2}}{hn^2} 
+ 3 L^2 \bdw d
\end{align*}
where $C >0$ is a universal constant. By choosing $\error =  n^{-2}$,
the first term dominates the second and third terms for $n = \Omega(e^{d})$. It follows that
\begin{align}
\label{eq:upper-bd1}
    \E \| \hat s^{\error}_{\bdw} - s^\ast \|^2_{\rho^\ast_\bdw} \le \frac{C d^4 (\alpha^2 \log n)^{d/2}}{n\bdw^{d/2+1}} \pth{\log \frac{n}{\bdw}}^3 + 3 L^2 \bdw d
\end{align}
where $C > 0$ is a different universal constant.
Finally, we optimize the bound~\eqref{eq:upper-bd1} over $\bdw$.
By choosing $\bdw = \inparen{\frac{d^3 (\alpha^2 \log n)^{d/2}}{L^2 n}}^{\frac{2}{d+4}}$, the two terms in~\eqref{eq:upper-bd1} are in the same order, and hence we obtain the desired bound by the change of measure in~\eqref{eq:change-of-measure-ub}:
$$\E \| \hat s^{\error}_{\bdw} - s^\ast \|^2_{\rho^\ast} \le C d \alpha^2 L^2 \pth{\log n}^{\frac{d}{d+4}} n^{-\frac{2}{d+4}}$$
for a universal constant $C > 0$.
% C d^3 e^d \alpha^2 L^2 n^{-\frac{2}{d+4}}\pth{\log n}^{\frac{d}{d+4}}.$$
\end{proof}

The following result generalizes~\prettyref{thm:l2-error-kde} to score function $s^*$ that is $(L,\beta)$-H\"older-continuous with $0<\beta \leq 1$, satisfying
$\|s^*(x_1) - s^*(x_2)\| \le L\|x_1 - x_2 \|^\beta$ for any $x_1, x_2 \in \R^d$.
% The case $\beta = 1$ recovers~\prettyref{thm:l2-error-kde}.
% (with some worsened constant factors).
The proof uses a more general argument based on the score bound of Gaussian mixture in \cite[Proposition 2]{polyanskiy2016wasserstein} in place of the more delicate argument (\prettyref{lem:bound-c3}) based on log-concavity that requires the Lipschitzness of the score.
For readability, we provide the proof of~\prettyref{thm:upper-bound-holder} separately in Appendix~\ref{sec:holder}.

\begin{theorem}
\label{thm:upper-bound-holder}
    Fix $d \ge 1$. Suppose we have $X_1,\cdots, X_n$ drawn i.i.d. from some $\rho^*$ which is a $\alpha$-subgaussian distribution on $\R^d$ with $(L, \beta)$-H\"older continuous score function for some $0 < \beta \le 1$. Setting
\begin{align*}
\error = n^{-2}
\qquad\text{ and }\qquad
\bdw = \inparen{\frac{d^{4-\beta} (\alpha^2 \log n)^{d/2} }{L^2 n}}^{\frac{2}{d+2\beta+2}}
% \bdw = \inparen{\frac{d^{4-\beta} (2\alpha^2 \exp\inparen{\inparen{C_2 + A^{2\beta}}L} \log n)^{d/2} }{L^2 n}}^{\frac{2}{d+2\beta+2}}
\end{align*}
for sufficiently large $n$, the score estimator~\eqref{Eq:KernelEst} satisfies
\begin{align*}
    \E\loss(\hat s_\bdw^\error, \rho^\ast) \le 
    C  d^\beta L^2 \alpha^{2\beta} (\log n)^{\frac{d\beta}{d+2\beta+2}} n^{-\frac{2\beta}{d+2\beta+2}}
    % C_1  L^2 d^\beta \exp\inparen{d L^{(1-\beta)/2}} \inparen{\frac{d^{4-\beta} (2\alpha^2 \exp\inparen{\inparen{C_2 + A^{2\beta}}L} \log n)^{d/2} }{L^2 n}}^{\frac{2\beta}{d+2\beta+2}}
\end{align*}
for a universal constant $C >0$.
\end{theorem}

\subsection{Minimax lower bound}

The following minimax lower bound, matching the upper bound in
Theorem \ref{thm:l2-error-kde} up to logarithmic factors, 
is proved in Appendix~\ref{Sec:ProofLB}.
The same argument is expected to extend straightforwardly to yield a matching lower bound for \prettyref{thm:upper-bound-holder} for 
the case of $\beta$-H\"older continuous scores.

\begin{theorem}
\label{thm:lb}
For any $d \geq 1$ and $\alpha>0$, there exist constants $c=c(d,\alpha)$ and
$L=L(d,\alpha)$ such that
    \begin{align}
        \inf_{\hat{s}} \sup_{\rho^\ast \in \mathcal{P}_{\alpha,L}}\E \loss(s, \rho^\ast) \ge c(d,\alpha) \, n^{-\frac{2}{d+4}}.
				\label{eq:scorelb}
    \end{align}
\end{theorem}

The above convergence rate can be interpreted by drawing analogy with classic results on estimating smooth densities in the H\"older class. It is well-known \cite{stone1982optimal,
stone1983optimal} that, for $m$-smooth densities supported on a  $d$-dimensional hypercube, the optimal rate (in mean squared $L^2$-error) of estimating the $r$th derivative is 
\begin{equation}
n^{-\frac{2(m-r)}{2m+d}}.
    \label{eq:stone}
\end{equation}
It is conceivable that estimating the score function (derivative of log density) is at least as hard as estimating the derivative of the density itself. Since the Lipschitz assumption of the score translates to twice differentiability of the density, 
we see that \prettyref{thm:lb} corresponds to \eqref{eq:stone}  with $m=2$ and $r=1$.

Furthermore, optimal estimation of $m$-smooth densities in squared error is attained by KDE with kernel chosen depending on the smoothness parameter $m$ and the optimal bandwidth $n^{-\frac{1}{d+2m}}$ 
\cite{stone1983optimal}. For $m=2$ this curiously coincides with the bandwidth choice $\sqrt{h}$ of the Gaussian kernel in Theorem \ref{thm:l2-error-kde}; for $m=\beta+1$, this coincides with Theorem \ref{thm:upper-bound-holder}.
By classical results in density estimation \cite{Tsybakov09}, for densities with smoothness parameter $m \leq 2$, positive kernels are optimal; however, for higher smoothness $m > 2$, kernels with negative parts must be used. For this reason, the methodology in the current paper based on Gaussian kernel may cease to be optimal if the score function is smoother than Lipschitz.
Determining the optimal rate of estimating $\beta$-H\"older scores with $\beta>1$ remains an open question.

The optimal rate $n^{-\frac{2}{d+4}}$ exhibits the typical ``curse of dimensionality'', suggesting that the optimal sample complexity reaching a given accuracy of score estimation must grow exponentially with the dimension. Rigorously establishing this is an interesting question that requires proving a sharper non-asymptotic lower bound for score estimation that applies to $d$ growing with $n$. We note that for estimating the density itself this was successfully carried out in \cite{mcdonald2017minimax}.

\section{Application to SGM}
\label{sec:application-sgm}

In this section, we discuss some implications of our results from Section~\ref{Sec:UpperBound} in the context of SGMs. 
We derive a finite-sample error bound for a KDE score estimator along the forward process, and then we plug in our result to existing guarantees for SGMs to deduce a final sample complexity result.
We first provide a brief review of a specific type of SGM called {\em Denoising Diffusion Probabilistic Modeling (DDPM)}; we refer the reader to~\cite{HJA20} for more details. 

Suppose our target distribution is $\nu$ on $\R^d$. In DDPM, we start with an Ornstein-Uhlenbeck (OU) process targeting $\gamma = \N(0, I_d)$:
\begin{align}
\label{eq:ou}
dX_t = -X_t \, dt + \sqrt{2} dW_t, \qquad X_0 \sim \nu_0 = \nu,
\end{align}
where $W_t$ is the standard Brownian motion in $\R^d$. Let $\nu_t = \text{Law}(X_t)$ be the distribution of $X_t$ along the OU flow, and let $s_t = \nabla \log \nu_t$ be the score function of $\nu_t$ \footnote{We note a change in the notation of $s_t$. Previously, $s_t$ denotes the score function of the Gaussian-smoothed target distribution with $t$ being the smoothing variance. In this section, $s_t$ is the score function of $\nu_t$ along the OU process starting from target distribution $\nu$. }.
We run the OU process~\eqref{eq:ou} until time $T > 0$, and then we simulate the backward (time-reversed) process, which can be described by the following stochastic differential equation:
\begin{align}
\label{eq:backwardOU}
  d\tilde{Y}_{t} = ( \tilde{Y}_{t} + 2 s_{T-t}(\tilde{Y}_{t})) dt + \sqrt{2} dW_t, \qquad \tilde{Y}_0 \sim \mu_0 = \nu_T.
\end{align}
By construction, if we start the backward process of~\eqref{eq:backwardOU} from $\tilde{Y}_0 \sim \mu_0 = \nu_T$, then we will have $\tilde{Y}_{t} \sim \mu_t = \nu_{T-t}$ for $0 \le t \le T$, and thus $\tilde{Y}_{T} \sim \mu_T = \nu$ is a sample from the target distribution.
However, in practice, we do not know
the score functions $\inparen{s_t}_{0 < t \le T}$.
Typically, we only assume we have independent samples from $\nu$, which we use to construct score estimators $\inparen{\hat{s}_t}_{0 < t \le T}$.
Then in the algorithm, we start the backward process~\eqref{eq:backwardOU} from $\tilde Y_0 \sim \gamma$ (the target distribution of forward process~\eqref{eq:ou} which is close to $\nu_T$ for large $T$), and we simulate the backward process in discrete time with score estimators $\inparen{\hat{s}_t}_{0 < t \le T}$ that we learn from samples. 
Let $\step > 0$ be the step size, and $K = \frac{T}{\step}$ so $T = K \step$; we assume $K \in \mathbb{N}$.
In each step, the DDPM algorithm performs the following update:
\begin{align}\label{eq:sgm}
\textrm{\textbf{(DDPM)}} \qquad\qquad\qquad y_{k+1} = e^{\step} y_k + 2(e^{\step}-1) \hat s_{\step (K-k)}(y_k) + \sqrt{e^{2\step}-1} z_k \qquad\qquad
\end{align}
where $\hat s_{\step k}$ is an approximation to the score function of $\nu_{\step k}$, and $z_k \sim \N(0,I_d)$ is an independent Gaussian random variable. 

We will apply the convergence result from~\cite{CCL+22}, who showed that the DDPM algorithm~\eqref{eq:sgm} returns a sample that is close in $\TV$ distance to $\nu$ up to the score estimation error and discretization error under some assumptions, including an error bound on the score estimation error.
We note that in \cite{CCL+22} the score estimator is learned from (and hence depends on) the sample, and their result is stated for a fixed score estimator. An inspection of their proof shows that the guarantees hold in expectation (with respect to the random sample), which we state below as Proposition~\ref{prop:convergence-sgm-ccl23}.
\begin{proposition}
\label{prop:convergence-sgm-ccl23}
Assume: 
\begin{enumerate}
    \item For all $t \ge 0$, the score $s_t = \nabla \log \nu_t$ is $L$-Lipschitz for some $1 \le L < \infty$;
    \item The target distribution $\nu$ has a finite second moment $m_2^2 \coloneqq \E_\nu [\|X\|^2] < \infty$;
    \item For $k = 1, \cdots, K$, the score estimator $\hat{s}_{k\step}(\cdot) = \hat{s}_{k\step}(\cdot;\X_n)$, which depends on the sample $\X_n$, has expected score estimation error $\E_{\X_n}\| s_{k\step} - \hat{s}_{k\step}\|_{\nu_{k\step}}^2 \le \epsilon_{\textnormal{score}}^2$; and
    \item The step size $\step = T/K$ satisfies $\step \lesssim 1/L$.
\end{enumerate}
At each time $t = k\step$, let $\rho_t$ be the law of the iterate $y_k$ of the DDPM algorithm~\eqref{eq:sgm} conditioned on the sample $\X_n$. 
Then it holds that:
   \begin{align}
    \label{eq:convergence-sgm-tv}
    \E_{\X_n}[\TV(\rho_T, \nu)] \lesssim e^{-T}\sqrt{\KL(\nu \| \gamma)} + \inparen{ L\sqrt{d\step} + Lm_2\step + \epsilon_{\textnormal{score}}}\sqrt{T}.
   \end{align}
\end{proposition}

\subsection{Score estimation along the OU flow} 
Given i.i.d.\ observations $X^{(1)}, \cdots, X^{(n)}$ from $\nu$, we need to estimate the score functions along the OU process, i.e.\ $s_t = \nabla \log \nu_t$ for any $t > 0$. Recall that following the OU flow~\eqref{eq:ou}, $X_t \sim \nu_t = \text{Law}(e^{- t} X_0) * \N(0, \tau(t) I_d)$
where $\tau(t) = 1-e^{-2t}$. When we start the OU flow with a finite set of observations $\{X^{(i)}\}_{i=1}^n$, i.e.\ the initial distribution of the flow is the empirical distribution $\hat\nu_0 = \frac{1}{n} \sum_{i=1}^n \delta_{X^{(i)}}$, the perturbed distribution at time $t$ is a Gaussian mixture $$\hat{\nu}_t = \frac{1}{n} \sum_{i=1}^n \N(e^{-t} X^{(i)}, \tau(t) I_d)$$
Its score function, $\nabla \log \hat{\nu}_t$, can be easily expressed in a closed-form. Using this closed-form score function for the sampling allows for a sampler without training. This may seem appealing, but using this score function in the backward SDE~\eqref{eq:backwardOU} will convert the noise to the empirical distribution $\hat{\nu}_0$, which means the model will memorize the training set and cannot generate novel samples. Therefore, we propose to use the regularized score function of $\hat{\nu}_t$:
\begin{align}
\label{eq:score-estimator-along-ou}
\hat{s}_t^{\error} = \frac{\nabla \hat\nu_t}{\max(\hat\nu_t, \error)}
\end{align}
for some $\error > 0$. The induced error in the closed-form score will enable the model to generalize.
Furthermore, we can analyze its performance by appealing to  similar techniques used in the proof of Theorem~\ref{thm:l2-error-kde}. We state the result as follows and provide the proof in Appendix~\ref{sec:pf-sgm}.

\begin{theorem}
\label{thm:score-estimator-conv-for-ou}
Fix $d \ge 1$. Assume $\nu$ is $\alpha$-subgaussian and has an $L$-Lipschitz score. Choose $\error = n^{-2}$. If the step size of DDPM~\eqref{eq:sgm} satisfies
$$\frac{1}{2} \log\inparen{1+\frac{\alpha^2 \log n}{n^{2/d}} } \lesssim \step \le \frac{1}{2}\log \inparen{1 + \frac{1}{4L-1}},$$
then at time $t \ge \step$, the squared $L^2(\nu_t)$ error of the estimator~\eqref{eq:score-estimator-along-ou} satisfies
\begin{align}
\label{eq:OU-score-error}
    \E \loss(\hat{s}_t^{\error}, \nu_t) = \E \| \hat{s}_t^{\error}  - s_t\|^2_{\nu_t}
    \lesssim \frac{1}{n}\frac{d}{(1-e^{-2t})} \inparen{\inparen{\frac{\alpha^2 \log n}{e^{2\step}-1}}^{d/2} + d}\inparen{\log \frac{n}{(2\pi (1-e^{-2t}))^{d/4}}}^{3}
\end{align}
where $\lesssim$ hides absolute constant, and the expectation is taken over the i.i.d.\ sample $X^{(1)}, \cdots, X^{(n)}$ from $\nu$.
\end{theorem}
Note for small $t$ and large $n$, the right-hand side above is $\tilde{O}(\step^{-d/2} (tn)^{-1})$. The bound decreases in $t$; this is because $\nu_t$ converges to the standard Gaussian as $t\to\infty$. 
In fact, our method shows that to reach $ \E \loss(\hat{s}_t^{\error}, \nu_t) \le \epsilon_{\textnormal{score}}^2$ for \textit{all} $t\geq \eta$, it suffices to have $\tilde{O}\left( \frac{d \alpha^d}{\step^{d/2+1} \epsilon_{\textnormal{score}}^2}\right)$ samples.
This is not obvious because, despite that both $\hat\nu_t$ and $\nu_t$ move closer to the same Gaussian as $t$ increases, it is unclear whether the score estimation error $\| \hat{s}_t^{\error}  - s_t\|^2_{\nu_t}$ is monotonically decreasing in $t$.\footnote{Even in the absence of the regularization parameter $\epsilon$, in which case 
$\| \hat{s}_t - s_t\|^2_{\nu_t}$ coincides with the relative Fisher information $\FI(\hat\nu_t \,\|\, \nu_t)$ in \prettyref{eq:FI}, 
it is still unclear whether this is monotone in $t$ because 
$\FI$ is convex in the first argument but not in the second. In comparison, $H^2(\hat \nu_t, \nu_t) = H^2(\hat \nu * \calN(0,(e^{2t}-1)I_d), \nu * \calN(0,(e^{2t}-1)I_d))$ is decreasing in $t$.
}
Nevertheless,  empirical Bayes techniques  (recall \prettyref{eq:program1}) allow us to control 
$\ell(\hat \nu_t, \nu_t)$ in terms of the Hellinger distance 
$H^2(\hat \nu_t, \nu_t)$ which satisfies data processing inequality \cite[Theorem 7.4]{PW-it} and hence decreasing in $t$. 
Furthermore, a simple application of Markov inequality shows that on a high probability event (with respect to $H^2(\hat \nu_\eta, \nu_\eta)$), the preceding bound on 
$  \loss(\hat{s}_t^{\error}, \nu_t) $ holds simultaneously for all $t\geq \eta$.

Combining Theorem~\ref{thm:score-estimator-conv-for-ou} with the previous convergence result for DDPM (Proposition~\ref{prop:convergence-sgm-ccl23}), we obtain the following sample complexity guarantee for DDPM driven by  the regularized score estimator \prettyref{eq:score-estimator-along-ou}:
\begin{corollary}
    Suppose the assumptions in Proposition~\ref{prop:convergence-sgm-ccl23} and Theorem~\ref{thm:score-estimator-conv-for-ou} hold. In order to have $\E\TV(\rho_T, \nu) \le \epsilon_{\TV}$, it suffices to run DDPM~\eqref{eq:sgm} with $\hat s_t = \hat{s}_t^{\error}$ in
    \prettyref{eq:score-estimator-along-ou} 
    for $T \asymp \log (\KL(\nu \| \gamma)/\epsilon_{\TV})$ and $\step \asymp \frac{\epsilon_{\TV}^2}{L^2 d}$, and have 
    $$n = \tilde{O}\inparen{\frac{d^{d/2+2}\alpha^d L^{d+2}}{\epsilon_{\TV}^{d+4}} }$$
    samples for score estimation.
\end{corollary}

\section{Discussion}
In this paper, we study the score estimation for subgaussian densities in $d$ dimensions with Lipschitz-continuous score functions. Under the score matching loss \prettyref{eq:LossFunction}, we establish a minimax lower bound at the rate of $n^{-\frac{2}{d+4}}$ using Fano's lemma. 
Applying techniques from empirical Bayes and smoothed empirical distribution as well as new insights enabled by the Lipschitzness of the true score, a regularized KDE score estimator using a Gaussian kernel with optimized bandwidth is shown to achieve this optimal rate up to logarithmic factors.
The convergence rate $n^{-\frac{2}{d+4}}$ suggests that an exponential increase in sample complexity is unavoidable as the dimension $d$ grows.

Within the SGM framework, particularly considering an OU flow as the forward process, if we start with $n$ independent observations, the score function along the OU flow has a closed-form (as the score of a Gaussian mixture). To improve generalization, we explicitly introduce a regularization term and analyze the performance of this estimator, leading to a sample complexity of $\tilde{O}\inparen{ \frac{d \alpha^d}{\step^{d/2 + 1}\epsilon_{\textnormal{score}}^2}}$ for score matching up to error $\epsilon_{\textnormal{score}}$, and a sample complexity of $ \tilde{O}\inparen{\frac{d^{d/2+2}\alpha^d L^{d+2}}{\epsilon_{\TV}^{d+4}} }$ for SGMs using this score estimator to reach a sampling error $\epsilon_{\TV}$. 
The exponential dependence of sample complexity on the dimension $d$ is fundamental to the nonparametric distribution class $\calP_{\alpha,L}$ we consider, which only assumes subgaussianity and score smoothness. There is a need to seek a meaningful distribution class whose sample complexity for score estimation has a milder dependency on $d$.

\section*{Acknowledgment}
Y.~Wu is grateful to Yandi Shen for helpful discussion on \prettyref{lem:bound-c2new} and 
to Sinho Chewi for discussion that motivated \prettyref{lmm:ic}.
 Y.~Wu is supported in part by the NSF Grant CCF-1900507 and an Alfred Sloan fellowship.

\appendix

\section{Details for proof of Theorem~\ref{thm:l2-error-kde}}

\subsection{Preliminary results}
\label{sec:preliminaries}

We present preliminary results which we will use.
We recall the following result from~\cite{SG2020}.

\begin{proposition}[{\cite[Theorem E.1]{SG2020}}]
\label{thm:sg2020}
Let $\rho_0$ and $\nu_0$ be two distributions on $\R^d$. Let $\rho_1 = \rho_0 * \N(0, I_d)$ and $\nu_1 = \nu_0 * \N(0, I_d)$. 
For $\error > 0$, let $s_{\rho_1}^{\error}$ and $s_{\nu_1}^{\error}$ denote the regularized score functions of $\rho_1$ and $\nu_1$ respectively.
If $\error \le (2\pi)^{-d/2}e^{-1/2}$, then
$$ \| s_{\rho_1}^{\error} - s_{\nu_1}^{\error} \|^2_{\rho_1} \le C d \max\left\{\left( \log \frac{(2\pi)^{-d/2}}{\error}\right)^{3}, \left|\log \H(\rho_1, \nu_1)\right|\right\} \H^2(\rho_1, \nu_1)$$
where $C$ is a universal positive constant.
\end{proposition}

Via a rescaling argument, we have the following generalization.

\begin{lemma}
\label{thm:c1-1}
Let $\rho$ and $\nu$ be two distributions on $\R^d$. Let $\bdw >0$, $\rho_\bdw = \rho * \N(0, \bdw I_d)$ and $\nu_\bdw = \nu * \N(0, \bdw I_d)$. For any $\error >0$, let $s_{\rho_\bdw}^{\error}$ and $s_{\nu_\bdw}^{\error}$ be the regularized score functions of $\rho_\bdw $ and $\nu_\bdw$ respectively.
If $\;0 < \error \le (2\pi \bdw)^{-d/2}e^{-1/2}$, then
$$ \| s_{\rho_\bdw}^{\error} - s_{\nu_\bdw}^{\error} \|^2_{\rho_\bdw} \le  \frac{Cd}{\bdw} \max\left\{\left( \log \frac{(2\pi \bdw)^{-d/2}}{\error}\right)^{3}, \left|\log \H(\rho_\bdw, \nu_\bdw)\right|\right\} \H^2(\rho_\bdw, \nu_\bdw),$$
where $C$ is a universal positive constant.
\end{lemma}
\begin{proof}
Let $X \sim \rho$ and $Y \sim \nu$, so $X_\bdw = X + \N(0, \bdw I) \sim \rho_\bdw$ and $Y_\bdw = Y + \N(0, \bdw I) \sim \nu_\bdw$. We can also write $X_\bdw$ and $Y_\bdw$ as
\begin{align*}
    &X_\bdw = \sqrt{\bdw} X' \qquad \text{where } ~ X'=\frac{X}{\sqrt{\bdw}} + \N (0, I_d) \sim \rho'  \\
    & Y_\bdw = \sqrt{\bdw} Y' \qquad \text{where } ~ Y'=\frac{Y}{\sqrt{\bdw}} + \N (0, I_d) \sim \nu'.
\end{align*}
Note that $\rho' = \law(X/\sqrt{\bdw})*\N(0, I_d)$ and $\nu' = \law(Y/\sqrt{\bdw})*\N(0, I_d)$. It follows from Proposition~\ref{thm:sg2020} that if $0 < \error' \le (2\pi)^{-d/2}e^{-1/2}$, then
$$\| s_{\rho'}^{\error'} - s_{\nu'}^{\error'} \|^2_{\rho'} \le C d \max\left\{\left( \log \frac{(2\pi)^{-d/2}}{\error'}\right)^{3}, \left|\log \H(\rho', \nu')\right|\right\} \H^2(\rho', \nu')$$
for some positive constant $C$.
By the relation of $\rho_\bdw$ and $\rho'$, we have
$\rho_{\bdw} (x) =  \bdw^{-d/2} \rho' (\frac{x}{\sqrt{\bdw}})$.
Thus we have the following relation of the score functions of $\rho_\bdw$ and $\rho'$, $$s_{\rho_\bdw}(x) = \frac{\nabla \rho_\bdw (x)}{\rho_\bdw (x)} = \frac{\bdw^{-d/2} \bdw^{-1/2} \nabla \rho'\left(\frac{x}{\sqrt{\bdw}}\right)}{\bdw^{-d/2} \rho' \left(\frac{x}{\sqrt{\bdw}}\right)} =  \frac{1}{\sqrt{\bdw}} s_{\rho'}\left(\frac{x}{\sqrt{\bdw}}\right)$$
and similarly for $\nu_\bdw$ and $\nu'$, $s_{\nu_\bdw}(y) = \bdw^{-1/2} s_{\nu'}\left(\frac{y}{\sqrt{\bdw}}\right)$.
The same holds for the regularized score functions, 
$$s_{\rho_\bdw}^{\error} (x) = \frac{1}{\sqrt{\bdw}}  s_{\rho'}^{\error'}\left(\frac{x}{\sqrt{\bdw}}\right) \qquad \textnormal{and }\qquad s_{\nu_\bdw}^{\error}(y) = \frac{1}{\sqrt{\bdw}} s_{\nu'}^{\error'}\left(\frac{y}{\sqrt{\bdw}}\right)$$
where $\error' = \bdw^{d/2}\error$. 
Therefore, if $0 < \error \le (2\pi \bdw)^{-d/2}e^{-1/2}$, i.e. $0 < \error' \le (2\pi)^{-d/2}e^{-1/2}$, then
\begin{align*}
    \| s_{\rho_\bdw}^{\error} - s_{\nu_\bdw}^{\error} \|^2_{\rho_\bdw} &= \int \rho_{\bdw}(\Tilde{x}) \;\| s_{\rho_\bdw}^{\error} (\Tilde{x}) - s_{\nu_\bdw}^{\error} (\Tilde{x})\|^2 d \Tilde{x} \\
    &= \frac{1}{\bdw} \int \rho' (x)\,\| s_{\rho'}^{\error'} (x)- s_{\nu'}^{\error'}(x) \|^2 dx \qquad\qquad (\text{by letting}\; \Tilde{x} = \sqrt{\bdw} x) \\
    &\stepa{\le} \frac{Cd}{\bdw}\, \max\left\{\left( \log \frac{(2\pi \bdw)^{-d/2}}{\error}\right)^{3}, \left|\log \H(\rho', \nu')\right|\right\} \H^2(\rho', \nu').
\end{align*}
where $(a)$ uses Proposition~\ref{thm:sg2020} and $\error' = \bdw^{d/2}\error$.
By the scale-invariance of the Hellinger distance,
\begin{align*}
    \H^2(\rho_\bdw, \nu_\bdw)= \H^2(\rho', \nu').
\end{align*}

Therefore, we obtain the desired result
$$\| s_{\rho_\bdw}^{\error} - s_{\nu_\bdw}^{\error} \|^2_{\rho_\bdw} \le \frac{Cd}{\bdw} \, \max\left\{\left( \log \frac{(2\pi \bdw)^{-d/2}}{\error}\right)^{3}, \left|\log \H(\rho_{\bdw}, \nu_{\bdw})\right|\right\} \H^2(\rho_{\bdw}, \nu_{\bdw}).$$
\end{proof}

\subsubsection{Hellinger convergence rate of smoothed empirical distribution}
\label{sec:bound-hellinger}
Another crucial ingredient of the proof is bounding the Hellinger distance between Gaussian-smoothed empirical distribution to the population.

\begin{lemma}
\label{lem:hellinger-bound}
    Let $d \ge 1$, $\bdw >0$ and $\alpha > 0$. Let $\rho^\ast \in \calP_{\alpha, L}$, which is an $\alpha$-subgaussian measure on $\R^d$ with an $L$-Lipschitz score $s^\ast$.
    Let $\rho^\ast_\bdw = \rho^\ast * \N(0, \bdw I_d)$. Let $\hat{\rho}$ be the empirical measure of an i.i.d.\ sample of size $n$ drawn from $\rho^\ast$ and $\hat{\rho}_\bdw = \hat{\rho} * \N(0, \bdw I_d)$. 
    Assume that $\bdw \leq \frac{1}{4L}$ and $h\leq \alpha^2$.
    Then
\begin{align}
    \E \H^2 (\hat \rho_\bdw,\rho^\ast_\bdw) \le 
     \frac{1}{n}\pth{\frac{C \alpha^2 \log n}{h}}^{d/2}
+ \frac{4d}{n},
\label{eq:smooth-hellinger}
\end{align}
where $C$ is some universal constant.
\end{lemma}

We note that convergence rate of smoothed empirical distribution has been well-studied in the literature (see, e.g., \cite{GGJ+2020,block2022rate}) under various metrics including different types of $f$-divergences and transportation distances, some of which exhibit a rich spectrum of behavior depending on the relationship between the smoothing parameter and the subgaussian parameter of the population. For example, in the setting of \prettyref{lem:hellinger-bound}, \cite[Proposition 2]{GGJ+2020} shows that
\begin{align}
    \E \TV (\rho^\ast_\bdw, \hat \rho_\bdw) \le \left( \frac{1}{\sqrt{2}} + \frac{\alpha}{\sqrt{\bdw}}\right)^{d/2} e^{\frac{3d}{16}} \frac{1}{\sqrt{n}},
    \label{eq:smooth-tv}
\end{align}
which holds for any $\alpha$-subgaussian $\rho^\ast$ without smoothness conditions on the score. 
Using the inequality $H^2/2 \leq \TV\leq H$ \cite[Sec.~7.3]{PW-it}, \prettyref{eq:smooth-tv}  implies that
$\E \H^2 (\hat \rho_\bdw,\rho^\ast_\bdw) \lesssim \frac{1}{\sqrt{nh^{d/2}}}$ up to constant depending on $d$ and $\alpha$.
Since the inequality $H^2 \leq \TV$ cannot be improved in general,\footnote{Note that we do have the  special structure  that $\rho^\ast_\bdw$ are $\hat \rho_\bdw$ are both Gaussian mixtures. The recent work \cite{jia2023entropic} shows that 
for Gaussian mixtures $H^2$ and $\KL$ are comparable. Whether $H$ and $\TV$ are comparable is posed as an open problem in \cite{jia2023entropic}.} this falls short of the desired 
bound of $\frac{1}{nh^{d/2}}$ in \prettyref{eq:smooth-hellinger} and hence the optimal rate of score estimation in \prettyref{thm:l2-error-kde}.
Another option is to apply the smoothed KL bound in \cite[Theorem 3]{block2022rate} and the fact that $H^2 \leq \KL$, leading to
$\E \KL (\hat \rho_\bdw\|\rho^\ast_\bdw) \leq \frac{C (\log n)^d}{n}$; unfortunately, examining the proof of this result shows that the constant $C$ is exponential in $1/h$.
In fact, our proof of \prettyref{lem:hellinger-bound} (notably, \prettyref{lmm:ratio} below)
crucially relies on the Lipschitzness of the score 
and directly deals with the Hellinger distance using a truncated second moment calculation. It is unclear whether \prettyref{lem:hellinger-bound} holds for all subgaussian distributions without smooth scores.

To show \prettyref{lem:hellinger-bound}, we start with an intermediate result.
\begin{lemma}
    \label{lmm:ratio}
    Let $p$ be a density on $\R^d$ whose score $s=\nabla\log p$ is $L$-Lipschitz. Let $p_h = p*\varphi_h$ where $\varphi_h$ is the density of $\N(0, hI_d)$.
    Then for all $y\in\reals^d$,
    \[
\frac{p}{p_h}(y) \leq  \exp(dLh/2).
    \]
\end{lemma}
\begin{proof}
    Let $Z\sim\calN(0,I_d)$.
    Then
    $p_h(y)=\Expect[p(y-\sqrt{h}Z)]$.
    We have
    \begin{align*}
\log \frac{p}{p_h}(y) 
= ~ & \log p(y) - \log \Expect[p(y-\sqrt{h}Z)] \\
\stepa{\leq} ~ &  \Expect[\log p(y) - \log p(y-\sqrt{h}Z)] \\
= ~ &  \Expect \int_{-\sqrt{h}}^0 
\Iprod{-Z}{s(y-uZ)} du \\
\stepb{=} ~ &  \Expect \int_{-\sqrt{h}}^0 
\Iprod{-Z}{s(y-uZ)-s(y)} du \\
\stepc{\leq} ~ &  \Expect \int_{-\sqrt{h}}^0 
L|u|\|Z\|^2 du = d Lh/2
\end{align*}
where (a) applies Jensen's inequality to the convex function $x \mapsto -\log x$;
(b) is because $\Expect[Z]=0$;
(c) applies Cauchy-Schwarz and 
the Lipschitzness of $s$.
\end{proof}

\begin{remark}
As we shall see below, the second moment calculation in bounding the expected square Hellinger distance requires controlling a quantity of the form $\int_{\reals^d} dy \frac{p}{p_h}(y)$ as $\poly(1/h)$.
The hope is that the convolution $p_h$ has a slightly heavier tail than that of $p$ such that the ratio $\frac{p}{p_h}(y)$ decays like a Gaussian of variance approximately $h$; this is exactly the case when $p$ is Gaussian.
While we cannot prove this in general,  \prettyref{lmm:ratio}  bounds the density ratio $\frac{p}{p_h}(y)$ by a constant independent of $y$. As such, additional truncation will need to be introduced  before passing to the second moment, which we do below. 

On the other hand, with more assumptions on the density $p$, it is possible to control $\int_{\reals^d} dy \, \frac{p}{p_h}(y)$ directly. For example, if $p$ is not only log-smooth but also strongly log-concave, then by analyzing the heat equation satisfied by $p_h$, one can show that $\int_{\reals^d} dy \, \frac{p}{p_h}(y) = O(h^{-d/2})$.     
\end{remark}

\begin{proof}[Proof of \prettyref{lem:hellinger-bound}]
Let $B\subset\reals^d$. 
Note that for any distributions $P$ and $Q$ with 
densities $p$ and $q$ on $\reals^d$,
\[
\H^2(p,q) = 
\int_{\reals^d} (\sqrt{p}-\sqrt{q})^2 
\leq
\int_B (\sqrt{p}-\sqrt{q})^2 
+ P(B^c)+Q(B^c)
\leq \int_B \frac{(p-q)^2}{q}
+ P(B^c)+Q(B^c).
\]
Thus, applying $\Expect\hat \rho_\bdw=\rho^\ast_\bdw$, we obtain
\begin{equation}
\E \H^2 (\hat \rho_\bdw,\rho^\ast_\bdw) 
\leq \int_B dy \frac{\Expect[(\rho^\ast_\bdw(y)- \hat \rho_\bdw(y))^2]}{\rho^\ast_\bdw(y)}
+ 2 \int_{B^c} \rho^\ast_\bdw.
    \label{eq:H1}
\end{equation}
Recall that $\varphi_h(x)=\frac{1}{(2\pi h)^{d/2}} e^{-\|x\|^2/(2h)}$ is the density of $\calN(0,h I_d)$. 
Let $X_i$ be i.i.d.~as $\rho^\ast$. 
Note that for each $y$,
\[
\hat \rho_\bdw(y) = \frac{1}{n} \sum_{i=1}^n
\varphi_h(y-X_i).
\]
Thus
$\Expect[\hat \rho_\bdw(y)]=\rho^\ast_\bdw(y)$ and 
$\Var[\hat \rho_\bdw(y)]=\frac{1}{n}
\Var(\varphi_h(y-X_1))$.
Note that 
$\varphi_h(x)^2 = (4\pi h)^{-d/2} \varphi_{h/2}(x)$.
So
\[
\Expect[\varphi_h(y-X_1))^2]
= (4\pi h)^{-d/2} 
\Expect[\varphi_{h/2}(y-X_1)]=
(4\pi h)^{-d/2} \rho_{h/2}^\ast(y)
\]
and
$\Var(\varphi_h(y-X_1))=
(4\pi h)^{-d/2} \rho_{h/2}^\ast(y) - (\rho_{h}^\ast(y))^2
$.

Combining the above with \prettyref{eq:H1}, we get
\begin{equation}
\E \H^2 (\hat \rho_\bdw,\rho^\ast_\bdw) 
\leq \frac{(4\pi h)^{-d/2}}{n}
\int_B dy \frac{\rho_{h/2}^\ast(y)}{\rho^\ast_\bdw(y)}
+ 2 \int_{B^c} \rho^\ast_\bdw.
    \label{eq:H2}
\end{equation}

    Next, choose $B = \mu+ [-a,a]^d$ where 
    $\mu$ is the mean of $\rho^*$,
    $a = \sqrt{C_0 \log n}$ for some $C_0$ to be specified.
Since 
$\rho^\ast$ is $\alpha$-subgaussian, 
$\rho_\bdw$ is 
$\sqrt{\alpha^2+h}$-subgaussian with the same mean  $\mu$.
Assuming $h\leq \alpha^2$, by union bound
\begin{equation}
\int_{B^c} \rho^\ast_\bdw
\leq 2d \exp\pth{- \frac{a^2}{4\alpha^2}}
\leq \frac{2d}{n}
    \label{eq:truncate}
\end{equation}
upon choosing $C_0=4\alpha^2$ and hence $a=2\alpha \sqrt{\log n}$.

For the first term in \prettyref{eq:H2}, 
since $h<\frac{1}{L}$,
 \prettyref{lem:bound-c3} in Appendix \ref{sec:pf-bound-c3} below implies that
the score of $\rho^\ast_{h/2}$ is $2L$-Lipschitz.
Applying \prettyref{lmm:ratio} 
to $p=\rho^\ast_{h/2}$ and $t=h/2$ yields
\begin{align}    
\label{eq:likelihood-ratio-h/2-h}
\int_B dy \frac{\rho_{h/2}^\ast(y)}{\rho^\ast_\bdw(y)}\leq
\text{vol}(B) \exp(dLh/4)
= (16\alpha^2 \log n)^{d/2}  \exp(dLh/4).
\end{align}
Combining everything we get
\[
\E \H^2 (\hat \rho_\bdw,\rho^\ast_\bdw) 
\leq \frac{1}{n}\pth{\frac{C \alpha^2 \log n}{h}}^{d/2}
+ \frac{4d}{n}.
\]   
\end{proof}

\subsection{Bounding the empirical Bayes regret}
\label{sec:pf-bound-c1}

\begin{lemma}\label{lem:bound-c1}
Assume $\rho^\ast$ is $\alpha$-subgaussian and has an L-Lipschitz score $s^\ast$. Let $\;0 < \error \le (2\pi \bdw)^{-d/2}e^{-1/2}$, and assume 
$$ \frac{\alpha^2\log n}{n^{2/d}} \lesssim \bdw \leq \frac{1}{4L}.$$
Then
$$\E\| \hat s^{\error}_{\bdw} - s^{\ast\error}_{\bdw} \|^2_{\rho^\ast_\bdw} \le \frac{Cd\inparen{C_{\bdw, d, \alpha}+d}}{n \bdw} \insquare{ \left( \log \frac{(2\pi \bdw)^{-d/2}}{\error}\right)^{3} + \, \log \frac{n}{C_{\bdw, d, \alpha}+d}}.$$
where $C >0$ is a universal constant and $C_{\bdw, d, \alpha} = \pth{\frac{\alpha^2 \log n}{\bdw}}^{d/2}$.
\end{lemma}
\begin{proof}
First, by Lemma~\ref{thm:c1-1}, we relate the quantity $\| \hat s^{\error}_{\bdw} - s^{\ast\error}_{\bdw} \|^2_{\rho^\ast_\bdw}$ in terms of the squared Hellinger distance between $\hat{\rho}_\bdw$ and $\rho^\ast_\bdw$ conditional on the samples:
$$\| \hat s^{\error}_{\bdw} - s^{\ast\error}_{\bdw} \|^2_{\rho^\ast_\bdw} \le \frac{Cd}{\bdw} \max\left\{\left( \log \frac{(2\pi \bdw)^{-d/2}}{\error}\right)^{3}, |\log \H(\rho^\ast_\bdw, \hat{\rho}_\bdw)|\right\} \H^2(\rho^\ast_\bdw, \hat{\rho}_\bdw)$$
where $C>0$ is a universal constant. Note that $|\log H| \leq \log \frac{2}{H}$ since $0 \leq H \leq \sqrt{2}$.
Taking expectation over $\X_n$, and using the simple bound $\max(a, b) \le a+b$ for $a,b\geq 0$, we have:
\begin{multline}
\label{eq:calcb12}
    \E \| \hat s^{\error}_{\bdw} - s^{\ast\error}_{\bdw} \|^2_{\rho^\ast_\bdw} 
    \le \frac{Cd}{\bdw} \left( \left( \log \frac{(2\pi \bdw)^{-d/2}}{\error}\right)^{3}\E \H^2(\rho^\ast_\bdw, \hat{\rho}_\bdw) + \frac{1}{2} \E \left[\H^2(\rho^\ast_\bdw, \hat{\rho}_\bdw) \,\log \frac{4}{\H^2(\rho^\ast_\bdw, \hat{\rho}_\bdw)}  \right]\right).
\end{multline}
By Lemma~\ref{lem:hellinger-bound}, we have that: If $\bdw \leq \frac{1}{4L}$ 
\begin{align}
\label{eq:ubhelinger}
    \E \H^2 (\rho^\ast_\bdw, \hat \rho_\bdw) \le \frac{1}{n}\pth{\frac{C \alpha^2 \log n}{h}}^{d/2}
+ \frac{4d}{n}.
\end{align}

On the other hand, since $x \mapsto  x \log \frac{4}{x}$ is concave, by Jensen's inequality
\begin{align*}
    \frac{1}{2} \E \left[ \H^2(\rho^\ast_\bdw, \hat{\rho}_\bdw)  \log \frac{4}{\H^2(\rho^\ast_\bdw, \hat{\rho}_\bdw)} \right]
    \le \frac{1}{2} \E \H^2(\rho^\ast_\bdw, \hat{\rho}_\bdw) \log \frac{4}{\E\H^2(\rho^\ast_\bdw, \hat{\rho}_\bdw)}.
\end{align*}
Recall that $x \mapsto  x \log \frac{4}{x}$ is increasing in $(0, 4/e)$. Let $$C_{\bdw, d, \alpha} \triangleq \pth{\frac{\alpha^2 \log n}{h}}^{d/2}.$$ If $\frac{C_{\bdw, d, \alpha}}{n} \le 4/e$, which can be satisfied when $\bdw \gtrsim \alpha^2 n^{-2/d}\log n$, then
\begin{align}
\label{eq:2nd}
    \frac{1}{2} \E \H^2(\rho^\ast_\bdw, \hat{\rho}_\bdw) \log \frac{4}{\E \H^2(\rho^\ast_\bdw, \hat{\rho}_\bdw)} \le  \frac{C_{\bdw, d, \alpha} + d}{2n}\, \log \frac{n}{C_{\bdw, d, \alpha}+d}.
\end{align}
Therefore, combining~\eqref{eq:calcb12},~\eqref{eq:ubhelinger} and~\eqref{eq:2nd}, we obtain the desired result
$$\E\| \hat s^{\error}_{\bdw} - s^{\ast\error}_{\bdw} \|^2_{\rho^\ast_\bdw}  \le \frac{Cd\inparen{C_{\bdw, d, \alpha}+d}}{n \bdw} \insquare{ \left( \log \frac{(2\pi \bdw)^{-d/2}}{\error}\right)^{3} + \, \log \frac{n}{C_{\bdw, d, \alpha}+d}}.$$
\end{proof}

\subsection{Bounding the regularization error}
\label{sec:pf-bound-c2}

The following result bounds the error introduced by the regularization parameter $\error$. 
Similar results appeared before in \cite[Theorem 3]{JZ09} for 1 dimension
and 
\cite[Lemma 4.3]{SG2020} for $d$ dimensions, the latter of which is not convenient to apply. 
Instead, we provide a self-contained improved version following the simple approach in \cite[Sec.~5.2]{shen2022empirical}.
\begin{lemma}
\label{lem:bound-c2new}
Let  $\rho^\ast$ be $\alpha$-subgaussian. 
Assume that $0 \leq \error \le (2\pi \bdw)^{-d/2}/e$ and 
$h \leq \alpha^2$. Then
$$ \| s^{\ast\error}_{\bdw} - s^\ast_{\bdw} \|^2_{\rho^\ast_\bdw} \le 
\frac{2\error}{h} (64 \alpha^2 \log n)^{d/2} 
\log \frac{1}{\error (2\pi h)^{d/2}}      + \frac{2d^{3/2}}{hn^2}.
$$
\end{lemma}

\begin{proof}
We first control the size of the score $s_h^\ast$. 
Let $U \sim \rho_*$ and $X=U + \sqrt{h}Z \sim \rho^\ast_h$, where $Z\sim \calN(0,I_d)$. Recall 
Tweedie's formula \prettyref{eq:tweedie0}, namely
    \begin{equation}
    s_h^\ast(x) = 
    \nabla\log \rho^\ast_h(x)= \frac{1}{h}  (\Expect[U|X=x]-x).
        \label{eq:tweedie}
    \end{equation}
Following \cite{JZ09,SG2020}, applying Jensen's inequality yields
\begin{align}
\|s^\ast_{\bdw}(x)\|^2 
 & \leq \frac{1}{h^2} 
\Expect[\|X-U\|^2 | X=x] \nonumber \\
& \leq \frac{2}{h}  \log  \Expect[\exp(\|X-U\|^2/(2h)) | X=x]
=  \frac{2}{h} \log \frac{1}{(2\pi h)^{d/2} \rho_h^\ast(x)}
\label{eq:scorebound}    
\end{align}
where the last equality is because the conditional density of $U$ given $X=x$ is 
$\frac{\exp\left(-\|x-u \|^2 / (2h) \right)\rho^*(u)}{(2\pi h)^{d/2} \rho_h^*(x)}.$

Next, as in the proof \prettyref{lem:hellinger-bound}, 
set $B=\mu+[-a,a]^d$, where $\mu=\Expect[U]=\Expect[X]$ and 
$a=\sqrt{64 \alpha^2 \log n}$. By the same subgaussian tail bound as in \prettyref{eq:truncate}, we have 
\begin{equation}
\prob{X \notin B} \leq \frac{2d}{n^4}.
    \label{eq:truncation2}
\end{equation}

Now we are ready to bound $\| s^{\ast\error}_{\bdw} - s^\ast_{\bdw} \|^2_{\rho^\ast_\bdw} $:
Recall from \prettyref{eq:def3} that $s_\bdw^{\ast\error} = \frac{\nabla \rho^\ast_\bdw}{\max(\rho^\ast_\bdw, \error)}$.
Then
\begin{align*}
    \| s^{\ast\error}_{\bdw} - s^\ast_{\bdw} \|^2_{\rho^\ast_\bdw}  
 \leq & \Expect[\|s^\ast_{\bdw}(X)\|^2\indc{\rho^\ast_\bdw(X) \leq \error}] \\   
 \leq & \underbrace{\Expect[\|s^\ast_{\bdw}(X)\|^2\indc{\rho^\ast_\bdw(X) \leq \error} \indc{X\in B}]}_{\text{(I)}} + \underbrace{\Expect[\|s^\ast_{\bdw}(X)\|^2  \indc{X\notin B}]}_{\text{(II)}}  .
\end{align*}
For the first term, applying \prettyref{eq:scorebound} we get
\begin{align*}
    \text{(I)} 
     = & \int_B dx \, \rho^\ast_\bdw(x)
    \|s^\ast_{\bdw}(x)\|^2 \indc{\rho^\ast_\bdw(x) \leq \error} \\
     \leq & 
        \frac{2}{h}  \int_B dx \,
   \rho^\ast_\bdw(x)  \log \frac{1}{(2\pi h)^{d/2} \rho_h^\ast(x)}  \indc{\rho^\ast_\bdw(x) \leq \error} \\
    \leq & 
    (2a)^d \frac{2\error}{h} \log \frac{1}{\error (2\pi h)^{d/2}}        
\end{align*}
where the last inequality follows form the fact that $t \log \frac{1}{t}$ is increasing on $t \in (0,1/e)$ and the assumption that $\error (2\pi h)^{d/2} < 1/e$.

For the second term, applying \prettyref{eq:tweedie}
\begin{align*}
    \text{(II)} 
     = & \frac{1}{h} 
     \Expect[\|\Expect[Z|X]\|^2  \, \indc{X\notin B}] \\
     \leq & \frac{1}{h} \sqrt{\Expect[\|Z\|^4] \,  \prob{X\notin B}} \\
    \leq & \frac{2d^{3/2}}{h n^2}.
\end{align*}
where the first  inequality applies Jensen's inequality and Cauchy-Schwarz inequality, 
and the second applies Jensen's inequality, $\Expect[\|Z\|^4]=2d+d^2$ 
and \prettyref{eq:truncation2}.
\end{proof}

\subsection{Lipschitzness of the Gaussian mixture score}
\label{sec:pf-bound-c3}

\begin{lemma}
\label{lem:bound-c3}
Assume $\rho_0$ is $L$-log-smooth. For $t \in (0, \frac{1}{2L}]$, $\rho_t =  \rho_0 * \N(0, tI_d)$ is $2L$-log-smooth.
\end{lemma}

\begin{proof}
For $X_0 \sim \rho_0$, let $X_t = X_0 + \sqrt{t} Z$ where $Z \sim \N(0,I_d)$ is independent, so $X_t \sim \rho_t$.
For $t > 0$, let $\rho_{0t}$ denote the joint distribution of $(X_0,X_t)$.
Let $\rho_{0 \mid t}(\cdot \mid y)$ denote the conditional density of $X_0$ given $X_t = y$.
Similarly, let $\rho_{t \mid 0}(\cdot \mid x)$ denote the conditional density of $X_t$ given $X_0 = x$, and note $\rho_{t \mid 0}(\cdot \mid x) = \N(x, t \, I_d)$ by definition.

By Tweedie's formula~\eqref{eq:tweedie0},
\begin{align}
\label{eq:score}
s_t(y) = \nabla \log \rho_t(y) = \frac{\E_{\rho_{0 \mid t}}[X \mid y] - y}{t}.
\end{align}
We then derive $\nabla^2 \log \rho_t (y)$:
Noting that 
\begin{align*}
    \nabla_y\, \rho_{0 \mid t} (x \mid y) &= \nabla\, \frac{\rho_{t \mid 0} (y\mid x) \rho_0(x)}{\rho_t(y)} \\
    &=\frac{\nabla \,\rho_{t \mid 0} (y\mid x) \rho_0(x)}{\rho_t (y)} - \frac{\rho_{t \mid 0} (y\mid x) \rho_0(x) \nabla \rho_t(y)}{ \rho_t^2(y)} \\
    &=  \rho_{0 \mid t} (x \mid y) \frac{x-y}{t} - \rho_{0 \mid t} (x \mid y) \nabla \log \rho_t(y),
\end{align*}
we obtain the gradient of the posterior mean
\begin{align*}
    \nabla \E_{\rho_{0 \mid t}}[X\mid y] &= \int \nabla \rho_{0 \mid t} (x \mid y) x^\top dx \\
    &= \E_{\rho_{0 \mid t}}\left[\frac{ (X-y)X^\top}{t} - \nabla \log \rho_t(y) X^\top \mid y \right] \\
    &\overset{\textnormal{\eqref{eq:score}}}{=} \frac{\E_{\rho_{0 \mid t}}[XX^\top \mid y]}{t} - \frac{\E_{\rho_{0 \mid t}}[X\mid y] \,\E_{\rho_{0 \mid t}}[X \mid y]^\top}{t}  \\
    &= \frac{\mathrm{Cov}_{\rho_{0 \mid t}}[X \mid y]}{t}.
\end{align*}
It follows that
\begin{align}
\label{Eq:HessRhot}
    -\nabla^2 \log \rho_t (y) = \frac{I_d}{t}- \frac{\mathrm{Cov}_{\rho_{0 \mid t}}[X \mid y]}{t^2}.
\end{align}
We now bound the covariance term. Suppose $\rho_0 \propto e^{-f}$. 
Recall that $\rho_{0 \mid t}(x \mid y) \propto e^{-f(x) - \frac{1}{2t} \| y-x \|^2}$, thus
$$-\nabla_x^2 \log \rho_{0 \mid t}(x\mid y) = \nabla_x^2 \left(f(x) + \frac{1}{2t} \| y-x \|^2\right) = \nabla^2f(x) + \frac{1}{t} I_d,$$
(note the derivative above is with respect to $x$).
Since $\nabla^2f(x) \preceq LI_d$, we have 
\[-\nabla_x^2 \log \rho_{0 \mid t}(x\mid y)\preceq \inparen{L + \frac{1}{t}} I_d.\]
This implies (see Lemma~\ref{Lem:Cov} below): For any $y \in \R^d$
$$\Cov_{\rho_{0 \mid t}}[X \mid y] \succeq \frac{1}{L + 1/t} I_d.$$
Therefore, we obtain an upper bound of the Hessian matrix~\eqref{Eq:HessRhot}:
\begin{align}
\label{eq:upbd}
-\nabla^2 \log \rho_t (y) \preceq \left( \frac{1}{t} - \frac{1}{t (tL+1)} \right) I_d = \frac{L}{tL + 1} I_d.
\end{align}
To get a lower bound, we note that since $\nabla^2 f(x) \succeq -LI_d$ for any $x \in \R^d$,
$$-\nabla_x^2 \log \rho_{0 \mid t}(x\mid y) \succeq (-L + \frac{1}{t}) I_d \succeq 0.$$
So for $t < \frac{1}{L}$, $\rho_{0 \mid t}(\cdot \mid y)$ is $(\frac{1}{t}-L)$-strongly log-concave, which implies
\[\Cov_{\rho_{0 \mid t}}[X \mid y] \preceq \frac{1}{1/t-L} I_d.\]
Therefore, for any $y \in \R^d$
\begin{align}
\label{eq:lowbd}
-\nabla^2 \log \rho_t (y) \succeq \left(\frac{1}{t}- \frac{1}{t(1-tL)}\right)I_d = - \frac{L}{1-tL}I_d.
\end{align}
Combining~\eqref{eq:upbd} and \eqref{eq:lowbd} gives
\[ - \frac{L}{1-tL}I_d \preceq -\nabla^2 \log \rho_t (y) \preceq \frac{L}{1+tL} I_d.\]
For $0 \le t < \frac{1}{L}$, $\frac{L}{1-tL} \ge \frac{L}{1+tL}$.
Therefore, $\rho_t$ is $\frac{L}{1-tL}$-log-smooth. If $t \le \frac{1}{2L}$, then we have $\frac{L}{1-tL} \le 2L$, so we conclude $\rho_t$ is $2L$-log-smooth for $0 \le t \le \frac{1}{2L}$.
\end{proof}

\begin{lemma}[\cite{brascamp1976extensions}]
\label{Lem:Cov}
Suppose a density $\rho$ on $\R^d$ satisfies $-\nabla^2 \log \rho(x) \preceq L I_d$ for all $x \in \R^d$. Then
    $$\Cov_\rho(X) \succeq \frac{1}{L} I_d.$$
\end{lemma}
This is a classical result; here we provide an alternate proof based on Fisher information calculation.

\begin{proof}
    Let $\nu = \N(m,C)$ be a Gaussian with the same mean $m = \E_\rho[X]$ and covariance $C = \Cov_\rho(X)$ as $\rho$.
    Note $-\nabla \log \nu(x) = C^{-1}(x-m)$.
    We can compute the relative Fisher information matrix of $\rho$ with respect to $\nu$ to be:
    \begin{align*}
        \tilde J_\nu(\rho) &\coloneqq \E_\rho\left[\left(\nabla \log \frac{\rho}{\nu}\right)\left(\nabla \log \frac{\rho}{\nu}\right)^\top \right] \\
        &= \E_\rho\left[\left(\nabla \log \rho\right)\left(\nabla \log \rho\right)^\top \right] + 
        \E_\rho\left[\left(\nabla \log \rho\right)\left(C^{-1}(x-m)\right)^\top \right]  \\
        &\qquad +  \E_\rho\left[\left(C^{-1}(x-m)\right)\left(\nabla \log \rho\right)^\top \right] + \E_\rho\left[\left(C^{-1}(x-m)\right)\left(C^{-1}(x-m)\right)^\top \right] \\
        &= \E_\rho[-\nabla^2 \log \rho] - C^{-1} - C^{-1} + C^{-1} C C^{-1} \\
        &\preceq LI_d - C^{-1}
    \end{align*}
    where the third equality above holds by integration by parts, and the last inequality holds by $L$-log-smoothness of $\rho$.
    Since $\tilde J_\nu(\rho) \succeq 0$, this implies $C^{-1} \preceq LI_d$ or equivalently 
    $C \succeq \frac{1}{L} I_d$, as desired.
\end{proof}

\subsection{Bounding the score error of the Gaussian smoothing}
\label{sec:pf-bound-c4}

\begin{lemma}
\label{lem:bound-c4}
Assume that $s^\ast$ is $(L,\beta)$-H\"older continuous for $0 < \beta \le 1$: For any $x_1, x_2 \in \R^d$
$$\|s^*(x_1) - s^*(x_2)\| \le L\|x_1 - x_2 \|^\beta.$$
Then
$$\| s^\ast_\bdw - s^\ast \|^2_{\rho^\ast_\bdw} \le L^2(\bdw d)^\beta.$$
\end{lemma}
\begin{proof}
    For $X \sim \rho^\ast$, let $Y = X+\sqrt{\bdw}Z$ where $Z \sim \N(0, I_d)$. Then $Y \sim \rho^\ast_\bdw$ and $\rho^\ast_\bdw(y) = \E \rho^\ast(y-\sqrt{\bdw}Z)$. Moreover, 
    \begin{align}
        s_\bdw^\ast(y) = \nabla \log \E \rho^\ast(y-\sqrt{\bdw}Z) = \frac{\E\insquare{s^\ast(y-\sqrt{\bdw}Z)\rho^\ast(y-\sqrt{\bdw}Z)}}{\E \rho^\ast(y-\sqrt{\bdw}Z)} = \E \insquare{s^\ast(X)\mid Y=y}.
        \label{eq:sh-rep}
    \end{align}
    Therefore, for any $y \in \R^d$
    \begin{align*}
        \| s^\ast_\bdw (y) - s^\ast(y) \|^2 \le \E \insquare{\| s^\ast (X) - s^\ast(y)\|^2 \mid Y=y} \le L^2 \E \insquare{\|X-y\|^{2\beta} \mid Y=y}.
    \end{align*}
    So $ \| s^\ast_\bdw - s^\ast \|^2_{\rho^\ast_\bdw} \le L^2 \bdw^{\beta} \E[\|Z\|^{2\beta}] \le L^2 (\bdw d)^\beta$. 
\end{proof}

\begin{remark}
    \label{rmk:scoreerror-highsmoothness}
A natural question is whether the score smoothing error can be improved if the true score has higher smoothness than Lipschitz (e.g.~$\beta$-H\"older for $\beta>1$, which is well-studied in nonparametric statistics \cite{Tsybakov09}.)
However, \prettyref{lem:bound-c4} as stated cannot be improved. For an example, consider $\rho^*=\calN(0,I_d)$ whose score is $s^*(x) = -x$. 
Then $\| s^\ast_\bdw - s^\ast \|^2_{\rho^\ast_\bdw} = \Theta(h)$.
% Then $\rho*=\calN(0,(1+h)I_d)$ and 
\end{remark}

% \begin{lemma}
% % \label{lem:bound-c4}
% Assume $s^\ast$ is $L$-Lipschitz. Let $\bdw \le 1/(4L)$. Then
% $$\| s^\ast_\bdw - s^\ast \|^2_{\rho^\ast_\bdw} \le 304 L^2\bdw d.$$
% \end{lemma}
% \begin{proof}
% By the score perturbation lemma~\cite[Lemma C.11]{LLT22a}: For any $x \in \R^d$
% $$\|s^\ast_\bdw (x) - s^\ast(x) \| \le 6L\sqrt{\bdw d} + 2L \bdw \|s^\ast(x)\|.$$
% Since $\|s^\ast(x)\| \le \|s^\ast(x) - s^\ast_\bdw(x)\| + \| s^\ast_\bdw(x)\|$, we have
% \begin{align}
% \label{eq:score-perturb}
%      (1-2Lh) \| s^\ast_\bdw (x) - s^\ast(x)\| \le 6L\sqrt{\bdw d }+ 2 L \bdw \|s^\ast_\bdw(x)\|.
% \end{align}
% By Lemma~\ref{lem:bound-c3}, $s^\ast_\bdw$ is $2L$-Lipschitz, i.e.\ $-2LI \preccurlyeq \nabla^2 \log \rho^\ast_\bdw \preccurlyeq 2LI$, so we have $-2dL \le \Delta \log \rho^\ast_\bdw (x)  \le 2dL$ for all $x \in \R^d$. 
% Note by integration by parts, we can write
% $\|s^\ast_\bdw\|^2_{\rho^\ast_\bdw} = -\E_{\rho^\ast_\bdw} \Delta \log \rho^\ast_\bdw$. 
% Therefore, 
% $$\|s^\ast_\bdw\|^2_{\rho^\ast_\bdw} \le 2dL.$$
% Combining the above with~\eqref{eq:score-perturb}, we obtain the desired result
% \begin{align*}
%     \| s^\ast_\bdw - s^\ast\|^2_{\rho^\ast_\bdw} \le \frac{72 L^2 \bdw d + 16 L^3 \bdw^2 d}{\inparen{1-2Lh}^2} \le 304 L^2 h d
% \end{align*}
% where the second inequality is due to $\bdw \le 1/(4L)$.
% \end{proof}

\section{Extensions to H\"older continuous  scores}
\label{sec:holder}

In this appendix we prove Theorem~\ref{thm:upper-bound-holder} on the estimation of  $\beta$-H\"older continuous score functions with $0<\beta\leq 1$.
The proof follows the same program of proving \prettyref{thm:l2-error-kde}, except that the key \prettyref{lmm:ratio} bounding the likelihood ratio of Gaussian convolutions, which relies on Lipschitzness of the score function, needs to be extended.
More specifically, 
\prettyref{lem:bound-c3}, which shows that the score function remains Lipschitz after convolving with sufficiently small Gaussian noise, applies the strong log-concavity of the posterior and the Brascamp-Lieb inequality. While it may be difficult to extend \prettyref{lem:bound-c3} to less smooth scores, 
it turns out that we can circumvent the score smoothness of Gaussian convolution in extending \prettyref{lmm:ratio}.
The following result bounds the score difference between the  smoothed and the original distributions by applying the score bound in \cite{polyanskiy2016wasserstein}.

\begin{lemma}
\label{lmm:ic}
    Let $s$ denote the score of $p$.
    Suppose $s$ is $(L,\beta)$-H\"older continuous for some $0<\beta\leq 1$.
    Let    $s_h$ be the score of $p_h=p * \N(0,h I_d)$.
    Then for any $y\in\reals^d$ and $h > 0$,
    \[
    \|s_h(y) - s(y)\| 
    % \leq 4 L (\|y\| + A)^\beta
    \leq 4 L (\|y-\mu\| + A)^\beta
    \]
    where 
    % $C$ is a universal constant and 
    $A = \Expect_{X\sim p}[\|X-\mu\|]$ and $\mu = \Expect_{X\sim p}[X]$.
\end{lemma}
\begin{proof}
Let $Y  = X + \sqrt{h} Z$, where $Z \sim N(0,I_d)$ and $X\sim p$ are independent.
Recall from \prettyref{eq:sh-rep}
that $s_h(y) = \Expect[s(X) \,|\, Y=y]= \Expect[s(y-\sqrt{h}Z) \,|\, Y=y]$.
    Then
    \begin{align*}
    \|s_h(y) - s(y)\| 
     \leq &  \,  \Expect[\|s(y-\sqrt{h} Z)) - s(y)\| \,|\, Y=y]\\
     \leq & \,L \, \Expect[\|\sqrt{h} Z\|^\beta \,|\, Y=y] \\
     = &\, L \, \Expect[\|y-X\|^\beta \,|\, Y=y]\\
     \stepa{\leq} & \,L \, ( \Expect[\|y-X\| \,|\, Y=y])^\beta\\
     = & \,L \, ( \Expect[\|(y-\mu)-(X-\mu)\| \,|\, Y-\mu = y-\mu])^\beta\\
     \stepb{\leq} & \,L (3 \|y-\mu\| + 4 \Expect[\|X-\mu\|])^\beta
    \end{align*}
    where 
    (a) is by Jensen's inequality; 
    (b) applies Proposition 2 (in particular, Eq.~(16)) in
    \cite{polyanskiy2016wasserstein} to the random variable $Y-\mu = (X-\mu) + \sqrt{h} Z$.
\end{proof}

The following lemma is a counterpart for 
\prettyref{lmm:ratio}:
\begin{lemma}
    \label{lmm:ratio-holder}
     Let $s$ denote the score of $p$.
    Suppose $s$ is $(L,\beta)$-H\"older continuous for some $0<\beta\leq 1$.
    Then for all $t > 0$,
    \begin{equation}
\label{eq:ratio-holder1}
\log \frac{p}{p_t}(y)
    \leq 
    % \frac{L}{1+\beta} 
    L\, (td)^{(1+\beta)/2}.
    \end{equation}
    Furthermore, for all $a>0$, $t > 0$, we have
  \begin{equation}
\label{eq:ratio-holder2}
    \log \frac{p_{a}}{p_{a+t}}(y)
    \leq 5
    L (td)^{(1+\beta)/2} + 4L\sqrt{td} (\|y-\mu\|^\beta+ A^\beta)
    \end{equation}
    where $A = \Expect_{X\sim p}[\|X-\mu\|]$ and $\mu = \Expect_{X\sim p}[X]$.
\end{lemma}

Later in the proof of Theorem~\ref{thm:upper-bound-holder}, we will apply 
\prettyref{eq:ratio-holder1} with $t=h$ (for change of measure, so that the likelihood ratio is bounded) or
\prettyref{eq:ratio-holder2} with 
$a=t=h/2$ (for bounding the smoothed empirical distribution, in which case $\|y-\mu\|\lesssim \sqrt{\log n}$ and 
$h=1/\poly(n)$ so it is also bounded).

\medskip

\begin{proof}
Following the proof of \prettyref{lmm:ratio}, we have
        \begin{align*}
\log \frac{p_{a}}{p_{a+t}}(y) 
= ~ & \log p_a(y) - \log \Expect[p_a(y-\sqrt{t}Z)] \\
{\leq} ~ &    \Expect \int_{-\sqrt{t}}^0 
\Iprod{-Z}{s_a(y-uZ)-s(y)} du \\
= ~ &  \Expect \int_{-\sqrt{t}}^0 
\Iprod{-Z}{s_a(y-uZ)-s(y-uZ)} du +
\Expect \int_{-\sqrt{t}}^0 
\Iprod{-Z}{s(y-uZ)-s(y)} du.
\end{align*}

Using the 
 $(L,\beta)$-H\"older continuity of $s$ and Jensen's inequality, the second term is upper bounded by
\[
\int_{-\sqrt{t}}^0 
L\Expect[\|Z\|^{1+\beta}] \, |u|^\beta du 
= \frac{L}{1+\beta} \Expect[\|Z\|^{1+\beta}] \, t^{(1+\beta)/2} 
\leq \frac{L}{1+\beta} d^{(1+\beta)/2} t^{(1+\beta)/2}.
 \]
(This proves \prettyref{eq:ratio-holder1} for $a=0$.)
 Applying \prettyref{lmm:ic}, the first term is bounded by:
 \begin{align*}
  4L \int_{-\sqrt{t}}^0 
 \Expect[\|Z\| (\|y-\mu-uZ\| + A)^\beta ] du 
 &\leq 4L \sqrt{t} \, 
 \Expect[\|Z\| (\|y-\mu\| + \sqrt{t} \|Z\| + A)^\beta ] \\
 &\leq 4L \sqrt{t} \left((\|y-\mu\|^\beta + A^\beta) \, \Expect[\|Z\|] + t^{\beta/2} \Expect[\|Z\|^{1+\beta}] \right) \\
 &\leq 4L \sqrt{td} (\|y-\mu\|^\beta + A^\beta) + 4L t^{(1+\beta)/2} d^{(1+\beta)/2}.
 \end{align*}
 Combining the two terms above yields the bound in~\prettyref{eq:ratio-holder2}.
\end{proof}

With the above lemma, we extend \prettyref{lem:hellinger-bound} on the smoothed
empirical distribution to H\"older-continuous scores.
\begin{lemma}
\label{lem:hellinger-bound-holder}
If $\rho^*$ is $\alpha$-subgaussian and $s^*$ is $(L,\beta)$-H\"older continuous for some $0 < \beta \le 1$, then for $h \le \frac{1}{4L}$ and $h \le \alpha^2$,
\begin{align}
    \E \H^2 (\hat \rho_\bdw,\rho^\ast_\bdw) \le
     \frac{1}{n}\pth{\frac{2\alpha^2 \log n}{h}}^{d/2} \exp\inparen{C_1 \sqrt{h} + C_2 \sqrt{\bdw} (\log n)^{\beta/2}}
    + \frac{4d}{n}
\label{eq:smooth-hellinger-holder}
\end{align}
where $C_1 = 9L \alpha^\beta d^{\frac{1+\beta}{2}}$ and $C_2 = 8 L \alpha^\beta d^{\frac{1+\beta}{2}}$.
\end{lemma}
\begin{proof}
The proof of \prettyref{lem:hellinger-bound-holder} follows that of \prettyref{lem:hellinger-bound}, except that at the step \prettyref{eq:likelihood-ratio-h/2-h} we apply the bound~\prettyref{eq:ratio-holder2} from~\prettyref{lmm:ratio-holder} with $a=t=\bdw/2$.
We bound $\|y-\mu\| \le 2\alpha \sqrt{d \log n}$ for all $y$ in the box $B = \mu + [-2\alpha \sqrt{\log n}, \, 2\alpha \sqrt{\log n}]^d$ where $\mu = \Expect_{X \sim \rho^*}[X]$, we bound $A = \Expect_{X\sim \rho^*}[\|X-\mu\|]
\leq \sqrt{\Expect_{X\sim \rho^*}[\|X-\mu\|^2]} 
\le \sqrt{d} \alpha$ since $\rho^*$ is $\alpha$-subgaussian,
and $h^{\frac{1+\beta}{2}} \le \sqrt{h} \alpha^\beta$ since $h \le \alpha^2$.
\end{proof}
%%%

% \subsection{Proof of Theorem~\ref{thm:upper-bound-holder}}
% \label{sec:proof-upper-bound-holder}

We are now ready to complete the proof of Theorem~\ref{thm:upper-bound-holder}:
\begin{proof}
Following the proof of Theorem~\ref{thm:l2-error-kde}, we perform a change of measure as in \prettyref{eq:change-of-measure-ub}, by applying the bound~\prettyref{eq:ratio-holder1} from~\prettyref{lmm:ratio-holder} to get
\begin{align}
\label{eq:change-of-measure-holder}
\E\loss(\hat s^{\error}_{\bdw}, \rho^\ast) 
=  \E \| \hat s^{\error}_{\bdw} - s^\ast \|^2_{\rho^\ast} 
\le \exp\inparen{L (hd)^{(1+\beta)/2}} \E \| \hat s^{\error}_{\bdw} - s^\ast \|^2_{\rho^\ast_\bdw}.
\end{align}
Since we will choose $h = 1/\poly(n)$, the exponential factor in the right-hand side above is bounded by a constant.
Following \eqref{eq:decomposition}, $ \E \| \hat s^{\error}_{\bdw} - s^\ast \|^2_{\rho^\ast_\bdw}$ can be decomposed into the same three terms. 
For the second term, the bound~\prettyref{eq:upper-bound-3} from~\prettyref{lem:bound-c2new} still applies as its proof only uses subgaussianity of $\rho^*$. 
We can bound the third term by \prettyref{lem:bound-c4} as follows:
\begin{align}
\label{eq:3rd-term-holder}
\| s^\ast_\bdw - s^\ast \|^2_{\rho^\ast_\bdw} \le L^2 (\bdw d)^\beta.
\end{align}
To bound the first term, we mimic the proof of Lemma~\ref{lem:bound-c1}, which involves two crucial steps: The first step applies Lemma~\ref{thm:c1-1}, which still holds since it does not rely on any assumption on the score function. The second step involves deriving a Hellinger rate between the Gaussian-smoothed empirical distribution to the population (analogous to \prettyref{lem:hellinger-bound}), which can be extended as \prettyref{lem:hellinger-bound-holder}. Since we will choose $h = 1/\poly(n)$, the bound in \prettyref{lem:hellinger-bound-holder} can be further bounded by
\begin{align}
    \E \H^2 (\hat \rho_\bdw,\rho^\ast_\bdw) \le
     % \frac{\exp\inparen{\inparen{C_2 +A^{2\beta}} Ld}}{n}\pth{\frac{2 \alpha^2 \log n}{h}}^{d/2}
     \frac{C_2}{n}\pth{\frac{\alpha^2 \log n}{h}}^{d/2}
   + \frac{4d}{n},
\end{align}
where $C_2$ is a universal constant. Then the rest of the proof for bounding the first term follows that of \prettyref{lem:bound-c1} and the same bound holds with a different constant:
$$\E\| \hat s^{\error}_{\bdw} - s^{\ast\error}_{\bdw} \|^2_{\rho^\ast_\bdw} \le \frac{C_3 d\inparen{C_{\bdw, d, \alpha}+d}}{n \bdw} \insquare{ \left( \log \frac{(2\pi \bdw)^{-d/2}}{\error}\right)^{3} + \, \log \frac{n}{C_{\bdw, d, \alpha}+d}}$$
where $C_3$ is a universal constant and 
$C_{\bdw, d, \alpha} = \inparen{\frac{\alpha^2 \log n}{\bdw}}^{d/2}$.
% $C_{\bdw, d, \alpha} = \inparen{\frac{2\alpha^2 \exp\inparen{\inparen{C_2 + A^{2\beta}}L} \log n}{\bdw}}^{d/2}$.

Similar to the proof of \prettyref{thm:l2-error-kde}, by choosing $\error = n^{-2}$, 
$$ \E \| \hat s^{\error}_{\bdw} - s^\ast \|^2_{\rho^\ast_\bdw} \le  
\frac{C_3 d^4 (\alpha^2 \log n)^{d/2}}{n\bdw^{d/2+1}} \pth{\log \frac{n}{\bdw}}^3 
% \frac{C_3 d^4 (2\alpha^2 \exp\inparen{\inparen{C_2 + A^{2\beta}}L} \log n)^{d/2}}{n\bdw^{d/2+1}} \pth{\log \frac{n}{\bdw}}^3 
+ L^2 (\bdw d)^{\beta}.$$
Optimizing the bound by choosing 
$\bdw = \inparen{\frac{d^{4-\beta} (\alpha^2 \log n)^{d/2} }{L^2 n}}^{\frac{2}{d+2\beta+2}}$ 
% $\bdw = \inparen{\frac{d^{4-\beta} (2\alpha^2 \exp\inparen{\inparen{C_2 + A^{2\beta}}L} \log n)^{d/2} }{L^2 n}}^{\frac{2}{d+2\beta+2}}$ 
and combining with the change of measure \eqref{eq:change-of-measure-holder}, we obtain the desired rate:
$$\E\loss(\hat s^{\error}_{\bdw}, \rho^\ast)  \le 
C  d^\beta L^2 \alpha^{2\beta} (\log n)^{\frac{d\beta}{d+2\beta+2}} n^{-\frac{2\beta}{d+2\beta+2}}$$
% C  L^2 d^\beta \inparen{\frac{d^{4-\beta} (2\alpha^2 \exp\inparen{\inparen{C_2 + A^{2\beta}}L} \log n)^{d/2} }{L^2 n}}^{\frac{2\beta}{d+2\beta+2}}$$
for some universal constant $C>0$.
\end{proof}

\section{Proof of Lower Bound (Theorem~\ref{thm:lb})}
\label{Sec:ProofLB}

\begin{proof}[Proof of Theorem~\ref{thm:lb}.]
The proof of \prettyref{thm:lb} follows that of standard minimax lower bounds for nonparametric density estimation, with a few adjustments made for scores.	
By scaling, we can assume without loss of generality that $\alpha$ is some constant.
Fix a reference density $f_0=\varphi$, the standard normal density in $d$ dimensions. We create a collection of perturbations to $f_0$ by modifying its values on the unit cube $D=[0,1]^d$. Fix some $\epsilon>0$ to be specified later. Let $m=\frac{1}{\epsilon}$ and assume $m$ is an integer. Fix some kernel $w: \reals \to \reals$ satisfying the following conditions:
\begin{itemize}
	\item $w$ is supported on $[0,1]$, with $\int_0^1 w(x) dx=0$.
\item $w$ is twice differentiable, with $\max\{\|w\|_\infty,\|w'\|_\infty,\|w''\|_\infty\} \leq C$ for some absolute constant $C$.
As a result, $w$ satisfies the periodic boundary conditions $w(0)=w(1)=w'(0)=w'(1)=0$.
\end{itemize}
(For example, a concrete choice is a sinusoid kernel, such as $w(x) = (1 - \cos(4\pi x))\indc{0\leq x < 1/2} + (\cos(4\pi x)-1)\indc{\frac{1}{2}\leq x \leq 1}$.)
We then extend $w$ to $\reals^d$ as follows: for every $x=(x_1,\ldots,x_d)$, $w(x) \triangleq \prod_{t=1}^d w(x_t)$.
Then $w$ is twice differentiable and supported on $D$, with bounded gradient and Hessian, satisfying the periodic boundary conditions $w(x)=0$ and $\nabla w(x) = 0$ for all $x\in \partial D$.

For every $i=0,\ldots,m-1$, let $x_i=\frac{i}{m}$.
For a multi-index $\bfi=(i_1,\ldots,i_d) \in \calI \triangleq \{0,\ldots,m-1\}^d$, let $x_{\bfi}=(x_{i_1},\ldots,x_{i_d})$ and $D_{\bfi} = x_{\bfi} + \epsilon D$.
Then $D = \cup_{\bfi\in\calI} D_{\bfi}$.
For each $b=(b_{\bfi}: \bfi\in\calI) \in\calB \triangleq \{0,1\}^{\calI}$, let
\[
f_{b}(x) = f_0(x) + \epsilon^2 \sum_{\bfi\in\calI} b_{\bfi} w\pth{\frac{x-x_{\bfi}}{\epsilon}}.
\]
Note that provided that $\epsilon$ is at most a constant,\footnote{Here and below all constants depend on $d$.} each $f_b$ satisfies the following:
\begin{itemize}
	\item $f_{b}$ is a valid density on $\reals^d$.
	\item $f_{b}$ is bounded from above and below by a universal constant on $D=[0,1]^d$ and $f_{b}=f_0$ outside of $D$. Thus $f_{b}$ is $\alpha$-subgaussian for some constant $\alpha$.
	\item $f_{b}$ is twice differentiable on $\reals^d$. Furthermore, 
	$\|\nabla f_b\|\triangleq \sup_{x\in\reals^d}\|\nabla f_b(x)\|_{\infty}$ and 
	$\|\nabla^2 f_b\|\triangleq \sup_{x\in\reals^d}\|\nabla^2 f_b(x)\|_{\infty}$ are at most a constant. 
	
	\item The score $	s_{b} \triangleq \nabla \log f_{b} = \frac{\nabla f_{b}}{f_{b}}$ 	satisfies 
	\begin{align}
	s_b(x)= \begin{cases}
-x		&  x \in D^c \\
\frac{\nabla f_0(x) + \epsilon b_{\bfi} \nabla w\pth{\frac{x-x_{\bfi}}{\epsilon}} }{f_0(x) + \epsilon^2 b_{\bfi} w\pth{\frac{x-x_{\bfi}}{\epsilon}} }	 & x \in D_{\bfi}\\
	\end{cases}
	\label{eq:scoreb}
	\end{align}
	Furthermore, $s_b$ is	$L$-Lipschitz on $\reals^d$ for some constant $L$. To see this, note that its Jacobian is
	$D s_b = \frac{\nabla^2 f_{b}}{f_{b}} - \frac{\nabla f_{b}\nabla f^\top_{b}}{f_{b}^2}$.
	Since $f_b$ is lower bounded by a constant on $D$ and both $\nabla f_b$ and $\nabla^2 f_b$ are entrywise upper bounded everywhere, $D s_b$ is bounded on $D$ and hence 
	$s_b$ is $L$-Lipschitz on either $D$ or $D^c$. This implies that $s_b$ is $L$-Lipschitz on $\reals^d$. (Indeed, for every $x\in D$ and $y\in D^c$, there exists $z \in \partial D$  such that $\|x-y\|=\|x-z\|+\|z-y\|$. So $\|s_b(x)-s_b(y)\| \leq \|s_b(x)-s_b(z)\|+\|s_b(z)-s_b(y)\| \leq L(\|x-z\|+\|z-y\|)=L\|x-y\|$.)
	
\end{itemize}

In view of the above properties, we have $\calF=\{f_{b}:b\in\calB\} \subset \calP_{\alpha,L}$. So
\begin{align*}
\inf_{\hat s} \sup_{f \in \calP_{\alpha,L}} \Expect_f \|\hat s-s_f\|_{L^2(\reals^d,f)}^2 
\geq  & ~  \inf_{\hat s} \sup_{b \in \calB} \Expect_{f_b} \|\hat s-s_b\|_{L^2(\reals^d,f_b)}^2 \\
= & ~  \inf_{\hat s} \sup_{b \in \calB} \Expect_{f_b} \|\hat s-s_b\|_{L^2(D,f_b)}^2 + \Expect_{f_b} \|\hat s-s_b\|_{L^2(D^c,f_b)}^2 \\
\overset{(i)}{=} & ~ \inf_{\hat s} \sup_{b \in \calB} \Expect_{f_b} \|\hat s-s_b\|_{L^2(D,f_b)}^2	\\
\gtrsim & ~ \inf_{\hat s} \sup_{b \in \calB} \Expect_{f_b} \|\hat s-s_b\|_{L^2(D)}^2. \numberthis \label{eq:scorelb-reduction}
\end{align*}
Here the infimum in $(i)$ is over estimators $\hat s (\cdot) = \hat s (\cdot;X_1,\ldots,X_n)$ such that $\hat{s}(x) = -x$ for $x \in D^c$, $\|\hat s-s_b\|_{L^2(D)}^2 \triangleq \int_D dx\|\hat s(x)-s_b(x)\|_2^2$ stands  for the unweighted squared $L^2$-norm on $D$, and the last inequality holds because $f_b\in \calF$ is uniformly lower bounded on $D$.

After these reductions, the proof proceeds by a standard application of Fano's inequality as follows. 
\paragraph{Separation of scores.} For any $b,b'\in\calB$, 
\[
\|s_b-s_{b'}\|_{L^2(D)}^2 = \sum_{\bfi\in\calI} \int_{D_{\bfi}} dx \|s_b(x)-s_{b'}(x)\|_2^2 \indc{b_\bfi \neq b_{\bfi}'}.
\]
For each $\bfi \in \calI$ such that $b_\bfi \neq b_{\bfi}'$, say $b_\bfi=1$ and $b_{\bfi}'=0$, applying \prettyref{eq:scoreb} yields
\begin{align*}
\int_{D_{\bfi}} dx \|s_b(x)-s_{b'}(x)\|_2^2 
= & ~ \int_{D_{\bfi}} dx \left\|\frac{\nabla f_0(x) + \epsilon \nabla w\pth{\frac{x-x_{\bfi}}{\epsilon}} }{f_0(x) + \epsilon^2 w\pth{\frac{x-x_{\bfi}}{\epsilon}} }
- \frac{\nabla f_0(x)}{f_0(x)} \right\|_2^2  \\
= & ~ \epsilon^d \int_{D} dy \left\|\frac{\nabla f_0(x_{\bfi} + \epsilon y) + \epsilon \nabla w(y)}{f_0(x_{\bfi} + \epsilon y) + \epsilon^2 w(y) }
- \frac{\nabla f_0(x_{\bfi} + \epsilon y)}{f_0(x_{\bfi} + \epsilon y)} \right\|_2^2  \\
= & ~ \epsilon^d \int_{D} dy \left\|\frac{\epsilon  f_0(x_{\bfi} + \epsilon y) \nabla w(y) - \epsilon^2 w(y) \nabla f_0(x_{\bfi} + \epsilon y)}{f_0(x_{\bfi} + \epsilon y)(f_0(x_{\bfi} + \epsilon y) + \epsilon^2 w(y)) } \right\|_2^2  \\
\asymp & ~ \epsilon^{d+2}
\end{align*}
where the last step applies again the facts that 
% $f_0$ is bounded from below and both $w$ and $\nabla w$ are bounded from above on $D$.
$f_0$ and $w$ are bounded from above and below on $D$ and $\|\nabla w\|$ is  bounded from above on $D$. So for sufficiently small $\epsilon$, the numerator of the integrand scales as $\Theta(\epsilon^2)$, whereas the denominator scales as $\Theta(1)$. 
Thus 
\[
\|s_b-s_{b'}\|_{L^2(D)}^2 \asymp \epsilon^{d+2} d_{\textrm H}(b,b') 
\]
where $d_{\textrm H}(b,b')  = \sum_{\bfi\in\calI} \indc{b_{\bfi} \neq b_{\bfi}'}$ is the Hamming distance.

Next, by the Gilbert-Varshamov bound (see e.g.~\cite[Lemma 2.9]{Tsybakov09}), there exists an exponentially large packing $\calB'\subset\calB$, whose minimum Hamming distance is linear in the dimension, namely, 
$|\calB'| \geq \exp(c_0 m^d)$ and $\min_{b \neq b\in\calB'} d_{\textrm H}(b,b') \geq c_0 m^d$ for some universal constant $c_0$. Recalling that $m=1/\epsilon$, we have
\begin{equation}
\min_{b \neq b\in\calB'} \|s_b-s_{b'}\|_{L^2(D)}^2 \asymp \epsilon^{2}.
\label{eq:score-sep}
\end{equation}

\paragraph{Bounding the KL radius.}
For every $b\in \calB$, the Kullback-Leibler divergence satisfies
\begin{align*}
\KL(f_b\|f_0)
&\leq ~ \chi^2(f_b \| f_0) \\
&= ~ \int_{\reals^d} dx \frac{(f_b(x)-f_0(x))^2 }{f_0(x)} \\
&= \int_{D} dx \frac{(f_b(x)-f_0(x))^2 }{f_0(x)}\\
&\asymp ~ \int_{D} dx (f_b(x)-f_0(x))^2  \\
&\leq~ \sum_{\bfi\in\calI} \int_{D_{\bfi}} dx \pth{\epsilon^2  w\pth{\frac{x-x_{\bfi}}{\epsilon}}}^2 \\
&\overset{(i)}{\asymp}~ \epsilon^4
\end{align*}
where $(i)$ is because 
$$\sum_{\bfi \in \calI}\int_{D_i} w^2\pth{\frac{x-x_{\bfi}}{\epsilon}}dx = \sum_{\bfi \in \calI} \epsilon^d\int_{D} w^2\pth{y}dy = \sum_{\bfi \in \calI} \epsilon^d\pth{\int_0^1 w^2\pth{y_i}dy_i }^d \asymp m^d \epsilon^d = 1.$$
Since the observations are i.i.d., we get $\max_{b \in \calB} \KL(f_b^{\otimes n}\|f_0^{\otimes n}) \lesssim n \epsilon^4$.
By choosing $\epsilon = c n^{-\frac{1}{d+4}}$ for sufficiently small constant $c$, we get $n \epsilon^4 \leq c |\calB'|$.
In view of \prettyref{eq:score-sep}, applying Fano's inequality 
(see e.g.~\cite[Corollary 2.6]{Tsybakov09})
yields 
\[
\inf_{\hat s} \sup_{b \in \calB'} \Expect_f \|\hat s-s_b\|_{L^2(D)}^2 \gtrsim n^{-\frac{2}{d+4}},
\]
which, together with  \prettyref{eq:scorelb-reduction}, implies the desired lower bound \prettyref{eq:scorelb}.

\end{proof}

\section{Proof of Upper Bound in DDPM (Theorem~\ref{thm:score-estimator-conv-for-ou})}
\label{sec:pf-sgm}

\begin{proof}[Proof of Theorem~\ref{thm:score-estimator-conv-for-ou}]  
By triangle inequality, we can decompose $L^2(\nu_t)$ error into two components
\begin{align}
\label{eq:triangle-ineq-SGM}
    \E \| \hat{s}_t^{\error}  - s_t\|^2_{\nu_t} \le 2\E \| \hat{s}_t^{\error}  - s_t^{\error}\|^2_{\nu_t} + 2 \| s_t^{\error}  - s_t\|^2_{\nu_t}.
\end{align}
The first term can be bounded in the same manner as the proof of Lemma~\ref{lem:bound-c1}, as follows. 
First, we bound $\| \hat{s}_t^{\error}  - s_t^{\error}\|^2_{\nu_t}$ by Lemma~\ref{thm:c1-1}:
If $0 < \error \le (2\pi \tau(t))^{-d/2}e^{-1/2}$, then
\begin{align}
    \| \hat{s}_t^{\error}  - s_t^{\error}\|^2_{\nu_t} \le  \frac{Cd}{\tau(t)} \max\left\{\left( \log \frac{(2\pi \tau(t))^{-d/2}}{\error}\right)^{3}, \left|\log \H(\hat{\nu}_t, \nu_t)\right|\right\} \H^2(\hat{\nu}_t, \nu_t).
\end{align}
Let $X \sim \nu_0$ and $X' \sim \hat \nu_0$. Note that
$$\H^2(\hat \nu_t, \nu_t) = \H^2(\law(X + \sqrt{e^{2t}-1}\,Z), \law(X' + \sqrt{e^{2t}-1}\,Z))$$
where $Z \sim \N(0, I_d)$ and it is independent of $X, X'$.
Hence by data processing inequality, $\H^2(\hat \nu_t, \nu_t)$ is decreasing in $t$. It follows that
\begin{align*}
    \H^2(\hat \nu_t, \nu_t) \le \H^2(\hat \nu_\step, \nu_\step)
\end{align*}
where $\step$ is the step size in DDPM~\eqref{eq:sgm}.

Note that since $X$ is $\alpha$-subgaussian, $\law(e^{-\step}X)$ is $(e^{-\step} \alpha)$-subgaussian.
Then by Lemma~\ref{lem:hellinger-bound} we have: If $\step \le \frac{1}{2}\log \inparen{1 + \frac{1}{4L-1}}$,
\begin{align}
    \E \H^2(\hat{\nu}_t, \nu_t) \le \frac{1}{n} \pth{\frac{\alpha^2 \log n}{e^{2\step}-1}}^{d/2} + \frac{4d}{n}.
\end{align}
Therefore, similar to~\eqref{eq:calcb12}-\eqref{eq:2nd}, we obtain
\begin{align}
\label{eq:first-term-SGM}
    \E \| \hat{s}_t^{\error}  - s_t^{\error}\|^2_{\nu_t} \le \frac{Cd}{\tau(t)} \frac{C_{\step, d, \alpha}+d}{n}\insquare{ \inparen{\log \frac{1}{\error (2\pi \tau(t))^{d/2}}}^{3} + \log \frac{n}{C_{\step, d, \alpha}+d}}
\end{align}
where $C_{\step, d, \alpha} = \pth{\frac{\alpha^2 \log n}{e^{2\step}-1}}^{d/2}$,
if $\step \gtrsim \frac{1}{2} \log\inparen{\frac{\alpha^2 \log n}{n^{2/d}}+1}$.

The second term in~\eqref{eq:triangle-ineq-SGM} can be bounded by Lemma~\ref{lem:bound-c2new}:
If $0 \leq \error \le (2\pi \tau(t))^{-d/2}/e$, then
\begin{align}
\label{eq:second-term-SGM}
\| s_t^{\error}  - s_t\|^2_{\nu_t} \le \frac{2\error}{\tau(t)} (64 e^{-2t} \alpha^2 \log n)^{d/2} 
\log \frac{1}{\error (2\pi \tau(t))^{d/2}} + \frac{2d^{3/2}}{\tau(t)n^2}.
\end{align}

Combining~\eqref{eq:first-term-SGM} and~\eqref{eq:second-term-SGM}, and taking $\error = n^{-2}$, for $n = \Omega(e^d)$ we have
\begin{align*}
    \E \| \hat{s}_t^{\error}  - s_t\|^2_{\nu_t} \le \frac{Cd}{\tau(t)} \frac{C_{\step, d, \alpha}+d}{n} \inparen{\log \frac{n}{(2\pi \tau(t))^{d/4}}}^{3}.
\end{align*}
\end{proof}

\bibliographystyle{amsalpha}
\bibliography{score}

\end{document}

%% file: preamble.tex
\usepackage{amssymb, amsmath, amsthm}
\usepackage{amsfonts,mathtools,mathrsfs,mathtools,color,bm}

\usepackage[bookmarksnumbered=true,
            bookmarksopen=true,
            colorlinks=true,
            citecolor=red,
            linkcolor=blue,
            anchorcolor=red,
            urlcolor=blue]{hyperref}
            
\newtheorem{theorem}{Theorem}
\newtheorem{proposition}{Proposition}
\newtheorem{lemma}{Lemma}
\newtheorem{corollary}{Corollary}

\theoremstyle{definition}

\newtheorem{remark}{Remark}

\newcommand{\polylog}{\mathrm{polylog}}
\newcommand{\poly}{\mathrm{poly}}

\newcommand{\cP}{\mathcal{P}}

\newcommand{\R}{\mathbb{R}} 
\renewcommand{\H}{H}
\newcommand{\N}{\mathcal{N}}
\newcommand{\E}{\mathbb{E}}

\newcommand{\loss}{\ell}

\renewcommand{\part}[2]{\frac{\partial #1}{\partial #2}}

\newcommand{\Var}{\mathrm{Var}}

\newcommand{\TV}{\mathrm{TV}}
\newcommand{\FI}{\mathsf{FI}}
\newcommand{\X}{\mathbf{X}}

\newcommand{\law}{\text{Law}}
\newcommand{\step}{\eta}
\newcommand{\error}{\varepsilon}
\newcommand{\bdw}{h}

\newcommand{\Cov}{\mathrm{Cov}}

\def\Var{\mathrm{Var}}
\def\Cov{\mathrm{Cov}}

\newcommand{\inparen}[1]{\left(#1\right)}

\newcommand{\insquare}[1]{\left[#1\right]}

\newcommand\numberthis{\addtocounter{equation}{1}\tag{\theequation}}

\usepackage{prettyref}
\newrefformat{eq}{(\ref{#1})}
\newrefformat{thm}{Theorem~\ref{#1}}
\newrefformat{sec}{Section~\ref{#1}}
\newrefformat{algo}{Algorithm~\ref{#1}}
\newrefformat{fig}{Fig.~\ref{#1}}
\newrefformat{tab}{Table~\ref{#1}}
\newrefformat{rmk}{Remark~\ref{#1}}
\newrefformat{clm}{Claim~\ref{#1}}
\newrefformat{def}{Definition~\ref{#1}}
\newrefformat{cor}{Corollary~\ref{#1}}
\newrefformat{lmm}{Lemma~\ref{#1}}
\newrefformat{lem}{Lemma~\ref{#1}}
\newrefformat{prop}{Proposition~\ref{#1}}
\newrefformat{app}{Appendix~\ref{#1}}
\newrefformat{ex}{Example~\ref{#1}}

\newcommand{\reals}{\mathbb{R}}

\newcommand{\Expect}{\mathbb{E}}

\newcommand{\prob}[1]{\mathbb{P}\left[#1\right]}

\newcommand{\indc}[1]{{\mathbf{1}\left\{{#1}\right\}}}

\newcommand{\pth}[1]{\left( #1 \right)}

\newcommand{\Iprod}[2]{\langle #1, #2 \rangle}

\newcommand{\bfi}{\mathbf{i}}

\newcommand{\KL}{\mathrm{KL}}

\newcommand{\stepa}[1]{\overset{(a)}{#1}}
\newcommand{\stepb}[1]{\overset{(b)}{#1}}
\newcommand{\stepc}[1]{\overset{(c)}{#1}}

\newcommand{\calB}{{\mathcal{B}}}

\newcommand{\calF}{{\mathcal{F}}}

\newcommand{\calI}{{\mathcal{I}}}

\newcommand{\calN}{{\mathcal{N}}}

\newcommand{\calP}{{\mathcal{P}}}

\newcommand{\calR}{{\mathcal{R}}}